\documentclass[12pt]{article}

\usepackage{mathrsfs}
\usepackage[T1]{fontenc}
\usepackage[latin1]{inputenc}
\usepackage{enumerate}
\usepackage{hyperref}
\usepackage{comment}
\usepackage{bbm}
\usepackage{stmaryrd}
\usepackage{latexsym}
\usepackage{amsfonts}
\usepackage{parskip}
\usepackage{bbm}
\usepackage{amsmath,amssymb,amscd}
\usepackage{yfonts}
\usepackage{dsfont}
\usepackage[dvips]{graphicx}
\usepackage{epsfig}
\usepackage{psfrag}
\usepackage[font={small}]{caption}
\usepackage{color}
\usepackage{float}
\usepackage{upgreek}
\usepackage{tikz}
\usepackage{tikz-cd}
\usepackage{relsize}

\DeclareFontFamily{U}{txsyc}{}
\DeclareFontShape{U}{txsyc}{m}{n}{
   <-> txsyc%
}{}
\DeclareFontShape{U}{txsyc}{bx}{n}{
   <-> txbsyc%
}{}
\DeclareFontShape{U}{txsyc}{l}{n}{<->ssub * txsyc/m/n}{}
\DeclareFontShape{U}{txsyc}{b}{n}{<->ssub * txsyc/bx/n}{}
\DeclareSymbolFont{symbolsC}{U}{txsyc}{m}{n}
\SetSymbolFont{symbolsC}{bold}{U}{txsyc}{bx}{n}
\DeclareFontSubstitution{U}{txsyc}{m}{n}
\DeclareMathSymbol{\df}{\mathrel}{symbolsC}{"42}
\DeclareMathSymbol{\fd}{\mathrel}{symbolsC}{"43}
\DeclareMathSymbol{\lJoin}{\mathrel}{symbolsC}{"58}
\DeclareMathSymbol{\rJoin}{\mathrel}{symbolsC}{"59}

\newcommand{\cB}{{\cal B}}
\newcommand{\cC}{{\mathbf C}}
\newcommand{\cD}{{\discret{D}}}

\newcommand{\cS}{{\cal S}}

\newcommand{\cX}{{\discret X}}

\newcommand{\CC}{\mathbb{C}}

\newcommand{\EE}{\mathbb{E}}

\newcommand{\LL}{\mathbb{L}}
\newcommand{\NN}{\mathbb{N}}
\newcommand{\PP}{\mathbb{P}}

\newcommand{\RR}{\mathbb{R}}

\newcommand{\ZZ}{\mathbb{Z}}

\newcommand{\fe}{\mathfrak{e}}

\newcommand{\fm}{\mathfrak{m}}
\newcommand{\fn}{\mathfrak{n}}

\newcommand{\fX}{\mathbb{X}}

\newcommand{\di}{\displaystyle}
\newcommand{\iy}{\infty}
\newcommand{\lt}{\left}
\newcommand{\me}{\medskip}
\newcommand{\na}{\nabla}
\newcommand{\pa}{\partial}
\newcommand{\ri}{\rightarrow}
\newcommand{\rt}{\right}
\newcommand{\sm}{\smallskip}

\newcommand{\wi}{\widetilde}
\newcommand{\wit}{\widehat}

\newcommand{\Ent}{\mathrm{Ent}}

\newcommand{\fo}{\forall\ }

\newcommand{\Id}{\mathrm{Id}}

\newcommand{\lve}{\lt\vert}

\newcommand{\rve}{\rt\vert}

\newcommand{\st}{\,:\,}

\newcommand{\un}{\mathds{1}}

\newcommand{\bq}{\begin{eqnarray*}}
\newcommand{\bqn}[1]{\begin{eqnarray}\label{#1}}
\newcommand{\eq}{\end{eqnarray*}}
\newcommand{\eqn}{\end{eqnarray}}
\newcommand{\wwtbp}{\par\hfill $\blacksquare$\par\me\noindent}
\newcommand{\thistitlepagestyle}{}

\newcommand{\ttsim}{\raise.17ex\hbox{$\scriptstyle\mathtt{\sim}$}}

\newtheorem{pro}{Proposition}
\newtheorem{proposition}{Proposition}
\newtheorem{cor}[pro]{Corollary}
\newtheorem{lem}[pro]{Lemma}

\newtheorem{theorem}[pro]{Theorem}

\newtheorem{remark}[pro]{Remark}
\newtheorem{lemma}[pro]{Lemma}

\renewcommand{\thepro}{\arabic{pro}}

\newenvironment{rem}
{\par\me\refstepcounter{pro}\noindent{\bf Remark \thepro\ }}
{\par\hfill $\square$\par\sm\noindent}

\newcommand{\proof}{\par\me\noindent\textbf{Proof}\par\sm\noindent}

\newcommand{\dBg}{{\discret{G}}} 
\newcommand{\Bg}{G}   
\newcommand{\dLg}{{\discret{L}}} 
\newcommand{\Lg}{L}   
\newcommand{\dBs}{{\discret{Q}}} 
\newcommand{\Bs}{Q}   
\newcommand{\dLs}{{\discret{K}}} 
\newcommand{\Ls}{K}   
\newcommand{\esP}{\mathbf{P}}
\newcommand{\esD}{\mathbf{D}}
\newcommand{\esL}{\mathbf{L}}
\newcommand{\esLd}{\ell}
\newcommand{\esF}{\mathbf{F}}
\newcommand{\esC}{\mathbf{C}}
\newcommand{\esc}{\mathbf{c}}
\newcommand{\esPe}{\esP_{\! \fe}}
\newcommand{\esFf}{\esF_{\hskip-.3mm\mathrm{f}}}

\newcommand{\Hg}{{\rm{H}}_{\beta}}
\newcommand{\Jg}{{\rm{J}}_{\beta}}

\newcommand{\Cp}{{C_{0}(\RR_{+})}}

 \renewcommand{\L}{\mathbf{L}}

    \def\P{{\mathbb P}}

    \def\E{{\mathbb E}}
    \def\N{{\mathbb N}}

    \def\i{\textnormal {i}}
    
    \def\R{{\mathbb R}}

\newcommand{\meKp}{{\mathbf{n}_{\beta}}}

\newcommand{\discret}[1]{\mathbb{#1}}
\newcommand{\cont}[1]{{#1}}

\newcommand{\Ks}{\discret{K}}

\newcommand{\eigL}{\mathbb{L}_k}

\newcommand{\Lnu}{{\L^{2}(\nu_{\beta})}}

\newcommand{\Lv}{\esL^2(\mu_{\beta})}
\newcommand{\Lpv}{\esLd^2(\fm_{\beta})}

\newcommand{\adjL}{\Lambda^*}

\newcommand{\Gate}[3]{#1\stackrel{#2}{\curvearrowright}#3}

\title{On a gateway between continuous
and discrete Bessel and Laguerre  processes \footnote{This work was done when the first named author was visiting Cornell University and he is grateful to the School of ORIE for its hospitality. Both authors would like to thank D.~Golberg, J.~Pander and G.~Samorodnitsky for stimulating discussions.}}
\author{ L.~Miclo  \thanks{Toulouse School of Economics, CNRS, IMT-CEREMATH, Manufacture des Tabacs, 21, Allée de Brienne, 31015 Toulouse cedex 6, France, miclo@math.cnrs.fr.}  \hskip2mm and  P.~Patie \thanks{School of Operations Research and Information Engineering, Cornell University, Ithaca, NY 14853, USA, \hbox{pp396@cornell.edu}}
}

 \date{\vbox{\copy0
\vskip1mm
 \copy1
}
 }

\begin{document}




\maketitle
\thistitlepagestyle
\abstract{
By providing instances of approximation of linear diffusions by birth-death processes, Feller \cite{Feller-Gen}, has offered an original path from the discrete world to the continuous one. In this paper, by identifying an intertwining relationship
between
squared Bessel processes and
some linear birth-death processes, we show that this connection is in fact more intimate and goes in the two directions.  As by-products, we identify some  properties enjoyed by the birth-death family that are inherited from
squared Bessel processes. For instance, these include  a discrete self-similarity property and
a discrete analogue of the beta-gamma algebra.
We proceed by explaining that the same gateway identity also holds for the
corresponding ergodic Laguerre semi-groups.  It follows again that the continuous and discrete versions are more closely related than thought before, and this enables to pass information from one semi-group to the other one.
}
\\

{\small
\textbf{Keywords: } Squared Bessel processes, linear birth-death processes,  intertwining, continuous and discrete scaling,  spectral decomposition.
\par
\vskip.3cm
\textbf{MSC2010:} Primary: 60J60, 	60J10.
 Secondary: 	35P05, 65C40
}
\tableofcontents


\section{Introduction}


In a celebrated paper  \cite{Feller-Gen}, Feller provides a connection between continuous and discrete state space Markov processes by showing rigorously  some diffusion approximations by birth-death Markov chains. Lamperti's work \cite{Lamperti-67b}                       can be seen as a direct continuation and extension of Feller's ideas, introducing and analysing the continuum mass limits of Galton-Watson processes also
  for heavy-tailed offspring distributions.  From these works emerge the following approximation result which relate two central objects of our work. The  semi-group  $\discret{Q}^{(\beta)}=(\discret{Q}^{(\beta)}_t)_{t\geq 0},\: \beta>0$,
of the linear birth-death process $\discret{X}^{(\beta)}=(\discret{X}^{(\beta)}_t)_{t\geq 0}$ on $\ZZ_+$, whose generator
is the following difference operator
\begin{equation}
  \discret{G}_{\beta}=(n+\beta) \pa_+ + n\pa_-,\quad n\in \ZZ_+,
\end{equation}
where $\pa_{\pm}g(n)=g(n\pm 1)-g(n)$,
is an approximation of the diffusion semi-group $\cont{Q}^{(\beta)}=(\cont{Q}^{(\beta)}_t)_{t\geq 0}$ of the (scaled by $2$)-squared Bessel process $\cont{X}^{(\beta)}=(\cont{X}^{(\beta)}_t)_{t\geq 0}$ of index $\beta-1$ on $[0,\infty)$, whose
generator on $\RR_+$  is given by
\bq
 \qquad \Bg_{\beta}&=&x\pa^2+\beta\pa, \quad x>0.\eq

More specifically, one has, with $\lim_{\epsilon \to 0_+} \epsilon{\lfloor n/\epsilon \rfloor}=x$, that
\[\lim_{\epsilon\to 0_+}\discret{Q}_{t/\epsilon}^{(\beta)} \cont{d}_{\epsilon} f (\lfloor n/\epsilon \rfloor)= \cont{Q}^{(\beta)}_t f(x)\]
where $\cont{d}_c f(x)=f(cx)$ is the dilation operator. One can show, by a classical tightness argument,  that the convergence holds in the sense of weak convergence of probability measures on
$\mathbf{D}([0,\infty))$, the Skorokhod space of c\`adl\`ag paths.

The aim of this paper is to reveal  that, in fact, the connection between these two processes
(or their semi-groups) is even more intimate. Indeed, we shall provide
a direct connection  inducing an immediate limiting procedure. To describe it,
we define, for a bounded function $g$ on $\ZZ_+$, the Markov kernel $\Lambda$ by
\begin{equation}\label{eq:Lambda-def}
\Lambda g (x) = \E\left[g({\rm{Pois}}(x))\right], \quad x\geq 0,
\end{equation}
where ${\rm{Pois}}(x)$ is a Poisson random variable of parameter $x$,
and,  for $f$ a bounded and measurable function on $\RR_+$, the Markov kernel $\adjL$ by
\begin{equation}\label{eq:tild-Lambda-def}
\adjL f(n)=\EE[f({\rm{Gam}}(n+\beta))],\quad n\in \ZZ_+,
\end{equation}
where ${\rm{Gam}}(n+\beta)$ is a standard gamma random variable with shape parameter $n+\beta$.
Throughout, for two linear operators, $A$ and $B$, the notation  $\Gate{A}{\Lambda}{B}$ stands
for the intertwining relationship $A\Lambda=\Lambda B$ which holds   on the   specified domain. We also denote by
$\esc_0(\ZZ_+)$ (resp.~$\esC_0(\RR_+)$) the space of measurable (resp.~continuous) functions on $\ZZ_+$ (resp.~$\RR_+$) vanishing at infinity. For a measure $\mu$, we define the Hilbert space $\esL^2(\mu)=\{f :\RR_+\mapsto \R \: \textrm{ measurable with } \int_{0}^{\infty}f^2(x) \mu(dx)<\infty\}$ and when $\mu$ is a discrete measure we write $\ell^2(\mu)$.
\begin{theorem}\label{thm:main}\label{semig}
For any $\beta\geq 0$, we have
\begin{eqnarray}
\Gate{\cont{Q}_t^{(\beta)}}{\Lambda}{\discret{Q}_t^{(\beta)}} \quad \textrm{ in $\esc_0(\ZZ_+)$ }
\textrm{  and  } \quad
\Gate{\discret{Q}_t^{(\beta)}}{\adjL}{\cont{Q}_t^{(\beta)}} \textrm{ in $\esC_0(\RR_+)$}.
\label{eq:gatewayLag}
 \end{eqnarray}
These relationships also hold in  $ \ell^2({\fm_{\beta}})$, with $\fm_{\beta}(n)\df \frac{(n+\beta-1)(n+\beta-2)\cdots \beta}{n!}, n\in \ZZ_+$ and on $\esL^2(\mu_{\beta})$, where $\mu_{\beta}(dx)\df \frac{x^{\beta-1}}{\Gamma(\beta)}dx, \: x>0$, respectively.
   \end{theorem}
   \par
   Each of the relationships in \eqref{eq:gatewayLag} is a Markov intertwining as it was introduced by Pitman and Rogers \cite{Pitman1981}. We call it a gateway as it relates  directly continuous and discrete Markov processes as an alternative of usual approximation procedures. Note that it can also be seen as a lattice quantization.  In fact, we shall show that $\Lambda : \ell^2({\fm_{\beta}}) \mapsto \esL^2(\mu_{\beta})$  is a quasi-affinity,
 i.e.~a  one-to-one, bounded with dense range linear operator.
  This combines with a  result of Douglas \cite{Douglas} yield that  the first intertwining identity in \eqref{eq:gatewayLag} can be lifted to a unitary equivalence between these semi-groups, implying in particular that both semi-groups are isospectral.

 As a by-product, this gateway enables to identify some new invariance properties for the birth-death chain
 that are inherited from the well-known symmetries of the squared Bessel semi-groups, see e.g.~\cite{Pitman1981b} and \cite{Going-Yor-99}.
 For instance, the following $\cont{d}$-self-similarity property, valid for any $\sigma,t>0$,
   \begin{equation}\label{eq:self-bessel}
     \Gate{\cont{Q}^{(\beta)}_{t}}{\cont{d}_\sigma}{\cont{Q}^{(\beta)}_{\sigma t}}
   \end{equation}
     has the following discrete analogue.
     For any $\sigma>0$,    define $\discret{D}_\sigma$ the signed kernel from $\ZZ_+$ to $\ZZ_+$ given by the following
     binomial formula
    \bq
    \fo n,m\in\ZZ_+,\qquad \discret{D}_\sigma(n,m)&\df&\binom{n}{m}\sigma^m(1-\sigma)^{n-m}.
     \eq
     This kernel is Markovian only for $\sigma\in[0,1]$. We also use the notation for any  bounded function $g$ on $\ZZ_+$ and $n\in \ZZ_+$, $
 \discret{D}_\sigma f(n)=\sum_{m=0}^n f(m)\discret{D}_\sigma (n,m)$.
     \begin{proposition} \label{prop:self} \label{prop:comut-Dd}
    For any $\sigma>0$, we have
    \begin{equation}\label{eq:dil-D}
    \Gate{\cont{d}_\sigma}{\Lambda}{\discret{D}_\sigma}
    \end{equation}
    and, for any $t>0$,
      \begin{equation}\label{eq:self-besseld}
     \Gate{\discret{Q}^{(\beta)}_{t}}{\discret{D}_\sigma}{\discret{Q}^{(\beta)}_{\sigma t}}.
   \end{equation}
 We say that $\discret{Q}^{(\beta)}$ is $\discret{D}$-self-similar.
    \end{proposition}
    We point out that there is an interesting literature devoted to the study of the discrete self-similarity property,
    see e.g.~\cite{Klebanov} for a recent survey.
    For instance,  Steutel and van Harn \cite{Steutel-vanHarn} introduced  the binomial thinning operator
    to define  discrete stable variable. It boils
    down to the operator $\discret{D}_c$ when
     $0<c<1$.
    We proceed by recalling that Carmona et al.~\cite{Carmona-Petit-Yor-98}, showed the following interesting intertwining relationship between squared Bessel semi-groups of different indexes
\begin{eqnarray*}
 \label{eq:beta-gamma}
 \Gate{Q^{(\alpha+\beta)}_t}{ {{B}}_{\beta,\alpha}}{Q^{(\beta)}_t}
\end{eqnarray*}
where  \[ {{B}}_{\beta,\alpha}f(x)=\E\left[f(x {\rm{B}}({\beta,\alpha}))\right] =
\frac{\Gamma(\alpha+\beta)}{\Gamma(\alpha)\Gamma(\beta)}\int_{0}^{1}f(xr)r^{\beta-1}(1-r)^{\alpha-1}dr\]
 that is  ${\rm{B}}({\beta,\alpha})$ is a beta variable of parameter $\beta,\alpha>0$. Considering the intertwining identity at
time $t=1$
and $x=0$, they recover the following identity from the so-called beta-gamma algebra
\begin{eqnarray*}
 \label{eq:spectral_expansion_Q_t}
{\rm{B}}({\beta,\alpha})\times {\rm{Gam}}({\beta+\alpha})\stackrel{(d)}{=}{\rm{Gam}}({\beta})
\end{eqnarray*}
where here and below, in such a distributional identity, the random variables are assumed to be independent.
By considering a beta mixture of the gateway relationship \eqref{eq:dil-D},  i.e.~$c={\rm{B}}(\alpha,\beta)$, we obtain the following discrete
analogue of Carmona et al.~\cite{Carmona-Petit-Yor-98} analysis.
\begin{proposition} \label{prop:beta-gam-d}
For any $\alpha,\beta,t>0$, we have  on $\esc_0(\ZZ_+)$
\begin{equation}\label{eq:int-dbeta-gamma}
\Gate{\discret{Q}^{(\alpha+\beta)}_t}{\discret{B}_{\beta,\alpha}}{\discret Q^{(\beta)}_t}
\end{equation}
and, in particular,  we have  the discrete analogue of the beta-gamma algebra
\begin{eqnarray*}
 \label{eq:spectral_expansion_Q_t}
 {\rm{B}}({\beta,\alpha}) \odot{\rm{Pois}}({\rm{Gam}}(\alpha+\beta))\stackrel{(d)}{=}{\rm{Pois}}({\rm{Gam}}(\beta))
\end{eqnarray*}
  where the variables are all considered independent and the  binomial thinning operation is defined by
$\alpha \odot X = \sum_{i=1}^{X}b_i(\alpha)$ where $X$ is a $\mathbb{Z}_+$-valued variable,
$\alpha \in (0, 1)$, and $(b_i)$ is a sequence of independent
identically distributed (iid) Bernoulli  variables of parameter $\alpha$ and independent of $X$.
\end{proposition}

  In Section \ref{sec:cons-bes}, we shall provide additional by-products of the gateway identity,
  in relation to  the spectral decomposition of these semi-groups  we will offer an original proof of
  the construction of  the Laguerre polynomials as the Jensen polynomials
   of the Bessel functions.
   Moreover, it also provides an exact simulation of squared Bessel processes.

 We now proceed by recalling that the $\cont{d}$-self-similarity of the squared Bessel semi-group entails
 that the family of linear operators $\cont{K}^{(\beta)}\df(\cont{K}^{(\beta)}_t)_{t\geq 0}$ defined,
 for any $t,x\geq 0$, by
 \begin{equation}\label{eq:def_Lag}
   \cont{K}^{(\beta)}_t f(x) = \cont{Q}^{(\beta)}_{e^t-1} \cont{d}_{e^{-t}}f(x)
 \end{equation}
 is a Feller semi-group on $[0,\infty)$. It is thus natural to  wonder whether  the family of linear operators $\discret{K}^{(\beta)}=(\discret{K}^{(\beta)}_t)_{t\geq 0}$ defined, for any $t\geq 0$, by
 \begin{equation}\label{eq:def_Lag}
   \discret{K}^{(\beta)}_t g(n) = \discret{Q}^{(\beta)}_{e^t-1} \discret{D}_{e^{-t}}g(n),
 \end{equation}
 is a discrete Markov semi-group. We have the following.
 \begin{theorem}\label{thm:mainK}
For any $\beta\geq 0$, $\discret{K}^{(\beta)}$ is the Feller semi-group on $\N$ of a birth-death chain. Moreover,  we have on $\esc_0(\ZZ_+)$
\begin{eqnarray}\label{eq:gateway-lag}
\Gate{\cont{K}_t^{(\beta)}}{\Lambda}{\discret{K}_t^{(\beta)}}. 
 \end{eqnarray}
\end{theorem}
\par
It turns out that the diffusive semi-group $(\Ls^{(\beta)})_{t\geq 0}$ is ergodic and its invariant (even reversible) probability measure $\nu_\beta$ is the gamma distribution
of shape parameter $\beta>0$:
\bq \fo x>0,\qquad\nu_\beta(dx)&\df&x^{\beta-1}e^{-x}\, \frac{dx}{\Gamma(\beta)}.\eq
The birth-death
semi-group $(\dLs^{(\beta)})_{t\geq 0}$ is equally ergodic and its invariant (even reversible) probability measure $\fn_\beta$ is the negative Bernoulli distribution
of parameters 1/2 and $\beta>0$:
\bq
\fo n\in\ZZ_+,\qquad \fn_\beta(n)&\df& 2^{-n-\beta}\frac{\Gamma(\beta+n)}{n!\Gamma(\beta)}.\eq
It follows that \eqref{eq:gateway-lag} can be interpreted into the $\esL^2$ sense, namely from $\esL^2(\nu_\beta)$ to $\ell^2(\fn_\beta)$.
By means of the gateway relationship, we will also present what could be seen as an isospectral approximation of the  Laguerre diffusions by birth-death Laguerre processes.
The intertwinings of Theorems \ref{thm:main} and \ref{thm:mainK} can be strengthened into a graph of intertwining relations,
which will be investigated  in a general Markovian framework, as well as  its applications on speeds of convergence to  equilibrium,
in a forthcoming paper \cite{intertwining}. However as an avant-go\^ut,  we state  at the end of Subsection
 \ref{iocadLp}
 some accurate estimates on the convergence in the entropy sense of the semi-group $\dLs^{(\beta)}$
toward its equilibrium $\fn_\beta$, for $\beta\geq 1/2$, deduced from a corresponding result for $\Ls^{(\beta)}$.
\par\me
The plan of the paper is as follows. The next section is devoted to the proof of Theorem \ref{thm:main} that is to  the main gateway relationships.
In Section \ref{sec:cons-bes}, we state and proof some by-products of these relationships which include the proof of Propositions \ref{prop:self} and \ref{prop:beta-gam-d}, the relationship between Laguerre polynomials and Bessel functions. It also contains the characterization of the product of the intertwining kernel with its adjoint as the squared Bessel semi-group itself considered at time $1$. This interesting observation is then used to provide an exact simulation  of the squared Bessel processes. Section \ref{cabadLp} focusses on the study of the continuous and discrete ergodic Laguerre semi-groups. It contains the proof of Theorem \ref{thm:mainK}, the spectral decomposition of the discrete Laguerre semi-group. The appendix contains the study of the discrete scaling operator as a contractive semi-group in the Hilbert space.


\section{Gateway between continuous and discrete Bessel  processes}\label{cabadBp}


The aim  of this Section is to prove  Theorem \ref{thm:main}. 
We shall in fact provide two different proofs which all rely on specific properties of the involved processes. We find worth detailing each of them as they  may be used in a different context.
The first one hinges on a gateway relationship  between the generators of the Bessel and linear birth-death processes for which the linearity of their coefficient plays an important role. The second one, which offers an alternative proof in the Banach space $\esc_0(\ZZ_+)$ setting,  is based on a connection that we establish between the Laplace transform of the two semi-groups which seems to find its root in the branching property of the two processes.
Let now  $\mu_{\beta}, \beta>0,$ be the measure on $(0,\iy)$ given by
\[\mu_{\beta}(dx)\df \frac{x^{\beta-1}}{\Gamma(\beta)}dx, \: x>0,\] where $\Gamma$ is the gamma function. Moreover, let $\fm_{\beta}$ be the measure on $\ZZ_+$ given by
\bqn{mbe}
\fm_{\beta}(n)&\df&\binom{n+\beta-1}{n}\df \frac{(n+\beta-1)(n+\beta-2)\cdots \beta}{n!}, \: n\in\ZZ_+,\eqn
which is an extension of the usual binomial coefficient defined for $\beta\in\NN$.


The first proof  is split into several intermediate results that we state below and postpone their proofs to the forthcoming subsections. To describe the strategy of the proof we need to introduce a few notation. For $\beta\in\RR$, consider the $(2\beta)$-squared Bessel diffusion generator on $\RR_+$ given by
\bq
\fo x\in\RR_+, \qquad \Bg_{\beta}&=&x\pa^2+\beta\pa.\eq

Next, define the tri-diagonal operator $\dBg_{\beta,\alpha}$ acting on
$\esF(\ZZ_+)=\RR^{\ZZ_+}$, the set of all real mappings defined on $\RR_+$, via, for $ f\in\esF(\ZZ_+),$ and $n\in \ZZ_+$,
\bq
\dBg_{\beta,\alpha} f(n)&\df &(n+\beta)f(n+1)-(2\alpha n+\beta\alpha )f(n)+\alpha^2nf(n-1)\eq
(for $n=0$, it is not necessary to define $f(-1)$, since it is multiplied by 0).\par
For $\alpha\in\RR$, introduce the mapping
\bq
\fe_{\alpha}\st \RR_+\ni x\mapsto e^{\alpha x}\in\RR_+\eq
and consider the operator $\na_{\alpha}$ defined by
\bqn{nalph}
\na_{\alpha}\st \cC^\iy(\RR_+)\ni f&\mapsto& \lt((\pa^n\fe_{\alpha})f(0)\rt)_{n\in\ZZ_+}\in \esF(\ZZ_+).\eqn
 We are ready to state our first result which relies on  formal computations on the linear operators $\Bg_{\beta}$ and $\dBg_{\beta,\alpha}$ where the domain of the generators  does not play an important role, for instance
we can let $\Bg_{\beta}$ act on $\cC^\iy(\RR_+)$, the space of infinitely continuously differentiable functions on $\R^+$.\par
\begin{lem} \label{lem:gate_gen_B}
  We have on $\cC^\iy(\RR_+)$
  \bqn{gnab}
\Gate{\dBg_{\beta,\alpha}}{\na_{\alpha}}{\Bg_{\beta}}.\eqn
 The operator $\dBg_{\beta,\alpha}$ is a Markov generator if and only if $\alpha=1$ and $\beta \geq 0$. We write simply $\fe\df\fe_1$, $\dBg_{\beta}\df\dBg_{\beta,1}$ and $\na\df\na_{1}$.
\end{lem}
We would like to replace $\na$ by a Markov kernel from $\ZZ_+$ to $\RR_+$.
Let us first describe heuristically the procedure we will follow.
We start by finding an operator $\Lambda$ from $\RR_+$ to $\ZZ_+$ which is in some sense an inverse of $\na$.
Multiplying both side of \eqref{gnab} by $\Lambda$, on the left and on the right, we get
\bqn{b2}
\Gate{\Lambda\na \Bg_{\beta}}{\Lambda}{\dBg_{\beta}\na\Lambda}\eqn
namely the new intertwining relation
\bqn{b3}
\Gate{\Bg_{\beta}}{\Lambda}{\dBg_{\beta}}.\eqn
Since $\na$ corresponds to differentiations, $\Lambda$ is obtained through integrations and more precisely it will turn out to be
a Markov kernel from $\RR_+$ to $\ZZ_+$. To  develop this program in a more rigorous way, we introduce furthe notation. First, let us denote by $\Lambda$  the Markov kernel defined, for a bounded function $g$ on $\ZZ_+$, by
\bqn{Lambda}
\Lambda g(x)&=&\EE[g({\rm{Pois}}(x))]=\frac{1}{\fe(x)}\sum_{n\geq 0} \frac{g(n)}{n!} x^n, \quad x\geq 0,
\eqn
where we recognize  ${\rm{Pois}}(x)$ as a Poisson variable of parameter $x$. Next,
consider $\esPe$ the vector space of functions on $\RR_+$ which can be written under the form $P/\fe$, where $P$ is a polynomial function
and $\esFf(\ZZ_+)$ is the subspace of functions from $\esF(\ZZ_+)$ which vanish except on a finite number of points from $\ZZ_+$. Finally, we say that a linear operator between two Banach spaces is a quasi-affinity if it is bounded, one-to-one with a dense range. We are ready to state the following which also contains  some results on the operator $\Lambda$ that will be used later.

\begin{lem}\label{lem:lambda_def}\label{lem:lambda_bounded}\label{Linj}
\begin{enumerate}[1)]
  \item $\Lambda : \esPe \mapsto \esFf(\ZZ_+)$  is bijective with inverse $\na$.
  \item Moreover, the Markov kernel $\Lambda$ transports the measure $\mu_{\beta}$ into $\fm_{\beta}$ and it can be extended into a quasi-affinity, still denoted by $\Lambda$, from $\esLd^2(\fm_{\beta})$ to $\esL^2(\mu_{\beta})$
with an operator norm  bounded by $1$.
  \item Similarly, $\Lambda$ transports any probability measure $\nu$ on $\RR_+$ into a probability measure $\mathfrak{n}$ on $\ZZ_+$ and, as above, it can extended to bounded operator with dense range from $\esLd^2(\mathfrak{n})$ to $\esL^2(\nu)$. It is a quasi-affinity  if  $\mathfrak{n}(n)\sim C e^{-2 n \mathfrak{g}(n)}$, $C>0$ and $\mathfrak{g(n)}=o(\ln n)$.
       \item Finally, $\Lambda : \esc_0(\ZZ_+) \mapsto \Cp$ is a quasi-affinity.
\end{enumerate}
\end{lem}

We proceed by extending the validity of \eqref{b3} outside $\esFf(\ZZ_+)$, which requires to consider appropriate closures. For this purpose,  we assume, from now on, that $\beta>0$. Then, since  the vector space $\esPe$ (resp.~$\esFf(\ZZ_+)$) is dense in $\esL^2(\mu_{\beta})$ (resp.~$\esLd^2(\fm_{\beta})$), we shall show that $\Bg_{\beta}$ is self-adjoint and positive in $\esL^2(\mu_{\beta})$ (resp.~$\esLd^2(\fm_{\beta})$) and by invoking Freidrichs theorem, we obtain the following. 
\begin{lemma}\label{lem:G_L2}
$\Bg_{\beta}$ (resp.~$\dBg_{\beta}$) can be extended into a densely defined, closed and  self-adjoint operator on $\esL^2(\mu_{\beta})$ (resp.~$\esLd^2(\fm_{\beta})$)
 with domain $\esD(\Bg_{\beta})$ (resp.~$\esD(\dBg_{\beta})$).
 \end{lemma}
This yields to the following.
\begin{lem} \label{lem:lambda_range}
We have $\Lambda(\esD(\dBg_{\beta}))\subset \esD(\Bg_{\beta})$ and
formula \eqref{b3} is valid on $\esD(\dBg_{\beta})$.
\end{lem}
The  intertwining relation \eqref{b3}  can be extended at the level of the semi-groups $\Bs^{(\beta)}$ and $\dBs^{(\beta)}$. Heuristically the result is clear: it is sufficient to exponentiate \eqref{b3}.
However, one must be a little more careful and the details are provided in Section \ref{sec:proof_main}. We proceed with the following result which gives a representation of the adjoint operator of $\Lambda$ in the Hilbert spaces $\esL^2(\mu_{\beta})$, which allows to obtain the second gateway relationship.

\begin{lem}\label{lem:tildeL}
For any $f\in\esL^2(\mu_{\beta})$, we have
\bq
\fo n\in\ZZ_+,\qquad \adjL f(n)&=&\EE[f({\rm{Gam}}(n+\beta))]=\int_0^{\iy} f(x)\frac{x^{n+\beta-1}}{\Gamma(n+\beta)}e^{-x}dx
\eq
where ${\rm{Gam}}(n+\beta)$ is a standard gamma random variable of parameter $n+\beta$.
\end{lem}

\subsection{Proof of the lemmas}
\subsubsection{Proof of Lemma \ref{lem:gate_gen_B}}  First, note that for $\alpha\in\RR$,  the mapping
$ \fe_{\alpha} : \RR_+\ni x\mapsto e^{\alpha x}\in\RR_+ $ can also be  seen as a multiplication operator on $\cC^\iy(\RR_+)$ (similarly, $x$ will stand for the identity mapping
$\RR_+\ni x\mapsto x\in\RR_+$, as well as for the associated multiplication operator).
With this interpretation, we have the non-commutation relation
\bq
\fe_{\alpha}\pa&=&\pa \fe_{\alpha} -\alpha\fe_{\alpha}
=(\pa-\alpha )\fe_{\alpha}.\eq
We deduce that
\bq
\fe_{\alpha}\Bg_{\beta}&=&\fe_{\alpha}(x\pa^2+\beta\pa)
=x (\pa-\alpha )\fe_{\alpha}\pa +\beta(\pa-\alpha )\fe_{\alpha}\\
&=&[x(\pa-\alpha )^2+\beta(\pa-\alpha )]\fe_{\alpha}\\
&=&[x\pa^2-2\alpha x\pa+\alpha^2x+\beta\pa-\beta\alpha]\fe_{\alpha}.\eq
On the other hand, for $n\in\ZZ_+$, the Leibniz rule yields
\bq
\pa^n x&=&\sum_{m=0}^n\binom{n}{m} (\pa^mx)\pa^{n-m}
=\sum_{m=0}^n\binom{n}{m} (\pa^mx)\pa^{n-m}
=x\pa^n+n\pa^{n-1}.\eq
It follows that\bq
\pa^n[x\pa^2-2\alpha x\pa+x\alpha^2+\pa-\alpha ]&=&x\pa^{n+2}+n\pa^{n+1}-2\alpha x\pa^{n+1}\\ && -2\alpha n\pa^{n}
+\alpha^2x\pa^n+\alpha^2n\pa^{n-1}+\beta\pa^{n+1}-\beta\alpha\pa^n.\eq
Next, for any differential operator $\pa^n$, denote $\pa^n_{\vert 0}$ the value taken by this operator at the point $0\in\RR_+$,
so that $\pa^n_{\vert 0}$ can be seen as a linear form $\cC^\iy(\RR_+)\ri\RR$.
In particular, we have from the previous computations,
\bqn{entrela}
\nonumber\pa^n\fe_{\alpha}L_{b\vert 0}&=&\lt((x\pa^{n+2}+n\pa^{n+1}-2\alpha x\pa^{n+1}-2\alpha n\pa^{n}
+\alpha^2x\pa^n)\fe_{\alpha}\rt)_{\vert 0}\\&& +\lt((\alpha^2n\pa^{n-1}+\beta\pa^{n+1}-\beta\alpha\pa^n)\fe_{\alpha}\rt)_{\vert 0}\\
\nonumber&=&\lt((n\pa^{n+1}-2\alpha n\pa^{n}
+\alpha^2n\pa^{n-1}+\beta\pa^{n+1}-\beta\alpha\pa^n)\fe_{\alpha}\rt)_{\vert 0}\\
&=&(n+\beta)(\pa^{n+1}\fe_{\alpha})_{\vert 0}-(2\alpha n+\beta\alpha)(\pa^{n}\fe_{\alpha})_{\vert 0}+\alpha^2n(\pa^{n-1}\fe_{\alpha})_{\vert 0}.
\eqn
Its interest is that
the identity \eqref{entrela} can be written under the form of an intertwining relation:
\bqn{ab}
\na_{\alpha} \Bg_{\beta}&=&\dBg_{\beta,\alpha}\na_{\alpha}.\eqn
\par
We proceed by remarking  that the off-diagonal entries of $\dBg_{\beta,\alpha}$ are non-negative as soon as $\alpha, \beta\geq 0$. As a consequence, for $\alpha, \beta\geq 0$, the operator $\dBg_{\beta,\alpha}$ is a Markov generator if and only if
$\dBg_{\beta,\alpha}\un_{\ZZ_+}=0$, where $\un_{\ZZ_+}$ is the mapping always taking the value 1 on $\ZZ_+$.
We obtain that
\bq
\fo n\in\ZZ_+,\qquad \dBg_{\beta,\alpha}\un_{\ZZ_+}(n)&=&(\alpha-1)^2n+\beta(1-\alpha).\eq
It leads us to the choice $\alpha=1$ and $\beta\geq 0$ and from now on, the Markov generator $\dBg_{b,1}$ (respectively $\na_1$ and $\fe_1$) will be denoted $\dBg_{\beta}$ (resp.\ $\na$ and $\fe$), so that the intertwining relation
\eqref{ab} can be written, for any $\beta\geq 0$, as
\bqn{bb}
 \Gate{\dBg_{\beta}}{\na}{\Bg_{\beta}}. \eqn
From now on, the Markov generator
$\dBg_{\beta}$ will be represented by the infinite tri-diagonal matrix $(\dBg_{\beta}(m,n))_{m,n\in\ZZ_+}\df (\dBg_{\beta}[\un_{\{n\}}](m))_{m,n\in\ZZ_+}$, given explicitly by
\bqn{dBg}
\fo m,n\in \ZZ_+,\qquad
\dBg_{\beta}(m,n)&=&\lt\{\begin{array}{ll}
m&\hbox{ if $n=m-1$}\\
-2m-\beta&\hbox{ if $n=m$}\\
m+\beta&\hbox{ if $n=m+1$}\\
0&\hbox{ otherwise.}
\end{array}
\rt.
\eqn
\par\sm
\par
The operators $\Bg_{\beta}$ and $\dBg_{\beta}$ can be extended into self-adjoint operators, with respect to some natural $\esL^2$ structures on their respective state spaces,
say $\esL^2(\mu_{\beta})$ and $\esLd^2(\fm_{\beta})$, with $\mu_{\beta}\Lambda=\fm_{\beta}$.
Passing to the adjoints in \eqref{b3} with respect to the corresponding Hilbert structures, we get
\bqn{dd1}
\Gate{\dBg_{\beta}^*}{\adjL}{\Bg_{\beta}^*} \eqn
i.e.
\bqn{dd2}
\Gate{\dBg_{\beta}}{\adjL}{\Bg_{\beta}}. \eqn
If the measures $\mu_{\beta}$ and $\fm_{\beta}$ had finite weight, the Markovianity of $\Lambda$ would imply that of $\adjL$.
In our situation their weight is infinite, nevertheless it will turn out that $\adjL$ is a Markovian kernel and our goal will be fulfilled.\par\sm

\subsubsection{Proof of Lemma \ref{lem:lambda_def}}
It is clear that $\na\st \esPe\ri \esFf(\ZZ_+)$ is bijective, since for any polynomial $P(x)\df \sum_{n=0}^N a_nx^n$, we have
$\na(\esPe)=(n! a_n)_{n\in \ZZ_+}$. Denote $\Lambda\st \esFf(\ZZ_+)\ri\esPe$ the inverse mapping of $\na$, so that
$\na \Lambda=\Id$ and $\Lambda \na=\Id$, where the identity operators in the right-hand side are on $\esFf(\ZZ_+)$ and $\esPe$ respectively.
It follows that \eqref{b2} and \eqref{b3} are satisfied, when they are applied to functions from $\esFf(\ZZ_+)$.
Note that for any polynomial function $P$, we have the exact (finite) expansion, for any $ x\in \RR_+$,
\bq
P(x)&=&\sum_{n\in\ZZ_+} \pa^n P(0)\frac{x^n}{n!}\eq
so that
the action of $\Lambda$ is given, for any $ g\df(g(n))_{n\in\ZZ_+}\in\esFf(\ZZ_+)$ and $ x\in\RR_+,$ by \bq
\Lambda g(x)&=&\frac{1}{\fe(x)}\sum_{n\geq 0} \frac{g(n)}{n!} x^n
=\EE[g({\rm{Pois}}(x))]
\eq
where we recall that ${\rm{Pois}}(x)$ is a Poisson variable of parameter $x$. In particular, $\Lambda$ can be seen  as a Markov kernel from $\RR_+$ to $\ZZ_+$, by extending the above formula to any bounded $g\in\esF(\ZZ_+)$.
\par
For the next assertion, let $n\in\ZZ_+$ be given, and writing $\mathbb{I}_n(p)=\delta_{np}, n,p\in \N,$ we observe that
\bq
\mu_{\beta}\Lambda \mathbb{I}_n&=&\frac1{\Gamma(\beta)}\int_0^{\iy } \PP({\rm{Pois}}(x)=n)\, \mu_{\beta}(dx)
=\frac1{\Gamma(\beta)}\int_0^{\iy } \frac{x^n}{n!}e^{-x} \, x^{\beta-1} dx\\
&=&\frac1{\Gamma(\beta)n!}\int_0^{\iy } x^{n+\beta-1}e^{-x}  \, x^{\beta-1} dx
=\frac{\Gamma(n+\beta)}{\Gamma(\beta)n!}
=\binom{n+\beta-1}{n}\\
&=&\fm_{\beta}(n).
\eq
Note that this computation justifies the normalization by $\Gamma(\beta)$ imposed on $\mu_{\beta}$.
Next, fix a bounded function $g\in\esF(\ZZ_+)$ (or just an element $g\in\esFf(\ZZ_+)$).
Since $\Lambda$ is Markovian, we can use the Cauchy-Schwarz inequality to get
\bqn{csL}
\fo x\in\RR_+,\qquad (\Lambda g(x))^2&\leq & \Lambda g^2(x),\eqn
and, hence, using the previous identity,
\bq
\mu_{\beta} (\Lambda g)^2&\leq & \mu_{\beta}\Lambda g^2 =\fm_{\beta} g^2.\eq
Thus, by density of $\esFf(\ZZ_+)$ in $\esLd^2(\fm_{\beta})$, $\Lambda$ can be uniquely extended as an operator from $\esLd^2(\fm_{\beta})$ to $\esL^2(\mu_{\beta})$
whose operator norm is bounded by $1$. Next, since plainly $\esFf(\ZZ_+) \subset \esLd^2(\fm_{\beta})$,   and, from the discussion above, we have that $\Lambda (\esF(\ZZ_+))=\esPe$, we deduce that $\Lambda$ has a dense range since  the vector space $\esPe$ is dense in $\esL^2(\mu_{\beta})$. It remains to  show that $\Lambda : \esLd^2(\fm_{\beta}) \mapsto \esL^2(\mu_{\beta})$ is one-to-one. To this end, since for any $n\in \ZZ_+,$
\bqn{mbetan}
 \fm_{\beta}(n)&=& \frac{(n+\beta-1)(n+\beta-2)\cdots \beta}{n!}
=\lt(1+\frac{\beta-1}{n}\rt)\cdots \lt(1+\frac{\beta-1}{1}\rt)\eqn
so there exists a constant $c_{\beta}>0$ depending on $\beta>0$ such that for $n$ large, we have
\bq
\fm_{\beta}(n)&\sim& c_{\beta}n^{\beta-1}.\eq
As a consequence, for any $f\in\esLd^2(\fm_{\beta})$, we can then find a constant $C_f>0$ depending on $f$ such that
\bq
\fo n\in\ZZ_+,\qquad \lve f(n)\rve&\leq & C_fn^{(1-\beta)/2}\eq
and it follows that
the mapping $F$ defined by
\bqn{boundf}
\fo z\in \CC,\qquad F(z)&\df& \sum_{n\in\ZZ_+} \frac{f(n)}{n!}z^n\eqn
defines an entire function. If $f$ is furthermore in the kernel of $\Lambda$, then we must have a.e.\ in $x\in\RR_+$,
\bq
0\ =\ \Lambda f(x)
\ =\ e^{-x}F(x)\eq
and thus $F=0$ on $\RR_+$. By Cauchy Theorem,
we deduce that $\fo n\in\ZZ_+, f(n) =0$
i.e.~$f=0$, which completes the proof of the second claim of the  Lemma. For the next one, let $\nu$ be a probability measure on $\RR_+$, then as $\Lambda$ is a Markov kernel, the identity $\mathfrak n g = \nu \Lambda g$ plainly defines a probability measure on $\ZZ_+$. Moreover, proceeding as above, we easily show that $\Lambda$ extends to a bounded linear operator from $\esLd^2(\mathfrak n)$ into $\esL^2(\nu)$ and, since $\esF(\ZZ_+) \subset \esLd^2(\mathfrak n) $ and the vector space $\esPe$ is dense in $\esL^2(\nu)$,  $\Lambda$ has also a dense range. 
 Finally, recalling the condition  $\mathfrak{n}(n)\sim C e^{-n \mathfrak{g}(n)}$, $C>0$, $\mathfrak{g}(n)=o(\ln n)$ and  the Stirling formula $n! \sim \sqrt{2 \pi n}e^{n \ln n -n}$, 
 and observing that for $f\in \esc_0(\ZZ_+)$, i.e.~$|f(n)|<C (1-\epsilon)^{n}$, for some $0<\epsilon<1$ and $C>0$, $F$, in \eqref{boundf}, defines an entire function with  $F(x)\leq C e^{(1-\epsilon)x}$ for large positive $x$, similar arguments than the one developed avove prove the quasi-affinity property of $\Lambda$ on $\esc_0(\ZZ_+)$.

\subsubsection{Proof of Lemma \ref{lem:G_L2}}
First, note that the vector space $\esPe$ is included into $\esL^2(\mu_{\beta})$, as well as its image by $\Bg_{\beta}$.
Note furthermore that $\esPe$ is dense in $\esL^2(\mu_{\beta})$.
Another important observation is that $\Bg_{\beta}$ is symmetric in $\esL^2(\mu_{\beta})$. Indeed, we can factorize $\Bg_{\beta}$ under the form $x^{1-\beta}\pa x^{\beta}\pa$, so that, for all $ f,g\in \esPe,$
\bq
 \langle g,\Bg_{\beta} f\rangle_{\mu_{\beta}}&=&\frac1{\Gamma(\beta)}\int_0^{\iy}
g(x) x^{1-\beta}(\pa x^{\beta}\pa) f(x)\, x^{\beta-1}dx\\
&=&\frac1{\Gamma(\beta)}\int_0^{\iy}
g(x) (\pa x^{\beta}\pa) f(x)\, dx\\
&=&\frac1{\Gamma(\beta)}\lt[x^{\beta}g(x)\pa f(x)\rt]_0^{\iy}-\frac1{\Gamma(\beta)}\int_0^{\iy }
\pa g(x) \pa f(x)\,  x^{\beta} dx\\
&=&-\frac1{\Gamma(\beta)}\int_0^{\iy}
\pa g(x) \pa f(x)\,  x^{\beta} dx\eq
where  in the last-but-one equality, we used integration by parts, and in the last equality, the limit
\bq
\lim_{x\ri\iy } x^{\beta}g(x)\pa f(x)&=&0\eq
valid for any $f,g\in\esPe$.
Since the expression $\int_0^{\iy }
\pa g(x) \pa f(x)\,  x^{\beta} dx$ is symmetric with respect to $f$ and $g$, we get the announced symmetry property.
It also appears that $\Bg_{\beta}$ is non-positive, in the sense that, for all $f\in\esPe,$
\bq
 \langle f,\Bg_{\beta} f\rangle_{\mu_{\beta}}&=&
-\frac1{\Gamma(\beta)}\int_0^{\iy }
 (\pa f(x))^2\,  x^{\beta} dx
\leq 0.\eq
These properties imply that $\Bg_{\beta}$ can be closed into a self-adjoint operator on $\esL^2(\mu_{\beta})$, called its Freidrichs extension, see e.g.\ the book of Akhiezer and Glazman \cite{MR615737}. 
\par\sm
A similar closure can be considered for $\dBg_{\beta}$. Indeed, recalling that $
\fo n\in\ZZ_+,\: \fm_{\beta}(n)=\binom{n+\beta-1}{n}$, it is immediate to check that $\esFf(\ZZ_+)$ is a dense subspace of $\esLd^2(\fm_{\beta})$, that the image of $\esFf(\ZZ_+)$ by $\dBg_{\beta}$
is included into $\esLd^2(\fm_{\beta})$, and that $\dBg_{\beta}$ is symmetric and non-positive.
Again, we keep denoting $\dBg_{\beta}$ its Freidrichs extension and let $\esD(\dBg_{\beta})$ stand for its domain.
\par

\subsubsection{Proof of Lemma \ref{lem:lambda_range}}
Consider $g\in\esD(\dBg_{\beta})$. By definition, we can find a sequence $(g_n)_{n\in\ZZ_+}$
of elements from the core $\esFf(\ZZ_+)$ such that we have in $\esLd^2(\fm_{\beta})$,
\bq
\lim_{n\ri\iy} g_n&=&g\\
\lim_{n\ri\iy} \dBg_{\beta} g_n &=&\dBg_{\beta} g. \eq
Since $\Lambda$ is a bounded operator from $\esLd^2(\fm_{\beta})$ to $\esL^2(\mu_{\beta})$,
the sequences $(\Lambda g_n )_{n\in\ZZ_+}$ and $(\Lambda \dBg_{\beta} g_n )_{n\in\ZZ_+}$  converge respectively toward $\Lambda g $ and $\Lambda \dBg_{\beta} g $.
Taking into account that for any $n\in\ZZ_+$, we have $\Lambda g_{n} \in\esPe$, we deduce that $\Lambda g\in\esD(\Bg_{\beta})$ and that
$\Bg_{\beta}\Lambda g=\Lambda \dBg_{\beta} g$.
This observation amounts to the announced results.

\subsubsection{Proof of Lemma \ref{lem:tildeL}}
Let $\adjL\st \esLd^2(\ZZ_+)\ri\esL^2(\RR_+)$ be the adjoint operator of $\Lambda\st \esL^2(\RR_+)\ri\esLd^2(\ZZ_+)$.
Relation \eqref{dd1} is obtained by passing to the adjoints in \eqref{b3}, with respect to the Hilbert structures of  $\esL^2(\mu_{\beta})$ and $\esLd^2(\fm_{\beta})$.
By self-adjointness of $\Bg_{\beta}$ and $\dBg_{\beta}$, we deduce \eqref{dd2}.
By considering the equality \bqn{lls}\langle f,\Lambda g\rangle_{\mu_{\beta}}&=&\langle \adjL f, g\rangle_{\mu_{\beta}}\eqn
for any non-negative compactly supported functions $f\in\esL^2(\mu_{\beta})$ and $g\in\esLd^2(\fm_{\beta})$,
we get that $\adjL$ preserves the non-negativity.
To see that $\adjL$ is an abstract Markov kernel, it would remain to check that
\bq
\adjL \un_{\RR_+}&=&\un_{\ZZ_+}
\eq
but this equality can not be deduced
from \eqref{lls} applied with $f=\un_{\RR_+}$ and $g=\un_{\ZZ_+}$, because the constant mappings $\un_{\RR_+}$ and
$\un_{\ZZ_+}$ do not belong to $\esL^2(\mu_{\beta})$ and $\esLd^2(\fm_{\beta})$ respectively.
Instead, we resort to a direct computation, showing that $\adjL $ is a Markov kernel from $\ZZ_+$ to $\RR_+$:
let $f\in\esL^2(\mu_{\beta})$ and $g\in\esLd^2(\fm_{\beta})$ be two bounded and compactly supported functions.
We have
\bq
\langle \adjL f, g\rangle_{\mu_{\beta}}&=&\langle f,\Lambda g\rangle_{\mu_{\beta}}\\
&=&\int_0^{\iy } f(x)\Lambda g(x)\, \mu_{\beta}(dx)\\
&=&\int_0^{\iy } f(x)\sum_{n\in\ZZ_+}g(n)\frac{x^n}{n!}e^{-x}\, \mu_{\beta}(dx)\\
&=&\sum_{n\in\ZZ_+}\frac{g(n)}{n!}\int_0^{\iy } f(x)x^ne^{-x}\, \mu_{\beta}(dx)\\
&=&\sum_{n\in\ZZ_+}g(n)\lt(\frac{\Gamma(\beta)}{\Gamma(n+\beta)}\int_0^{\iy } f(x)x^ne^{-x}\, \mu_{\beta}(dx)\rt)\fm_{\beta}(n)
\eq
(the sums are in fact finite, so there is no problem of exchange of integral and sum).
Since this is true for any $g\in \esFf(\ZZ_+)$, we deduce that
\bq
\fo n\in\ZZ_+,\qquad
\adjL f(n)&=&\frac{\Gamma(\beta)}{\Gamma(n+\beta)}\int_0^{\iy } f(x)x^ne^{-x}\, \mu_{\beta}(dx)\\
&=&\int_0^{\iy } f(x)\frac{x^{n+\beta-1}}{\Gamma(n+\beta)}e^{-x}\,dx.\eq
\par
To get the validity of this formula for all $f\in\esL^2(\mu_{\beta})$,
we recall that $\esFf(\ZZ_+)$ is dense in $\esL^2(\mu_{\beta})$ and   $\adjL $ is a bounded operator.

\subsection{End of proof of Theorem \ref{thm:main}} \label{sec:proof_main}
We have now all the ingredients to complete the proof of Theorem \ref{thm:main} both in the Hilbert and Banach space settings. We point out that although the proof of the gateway relation in $\esc_0(\ZZ+)$ could be obtained by following a similar line of reasoning, we present, in this case, another proof in the next subsection which is based on the expression of the Laplace transform of the involved semi-groups.

\subsubsection{The Hilbert space case}

  First, since, from Lemma \ref{lem:G_L2}, the operator $\Bg_{\beta}$ is self-adjoint in the Hilbert space $\esL^2(\mu_{\beta})$, functional calculus can be used to define for any $t\geq 0$, $Q_t^{(\beta)}\df\exp(-t\Bg_{\beta})$.
The fact that $\Bg_{\beta}$ is non-positive implies that the spectrum of $\Bg_{\beta}$ is non-positive, so that for any $t\geq 0$, the spectrum of $\Bs_t^{(\beta)}$
is included into $(0,1]$ and in particular $\Bs^{(\beta)}_t\st \esL^2(\mu_{\beta})\ri\esL^2(\mu_{\beta})$ is a bounded operator.
It is well-known that the semi-group $\Bs^{(\beta)}\df(\Bs_t^{(\beta)})_{t\geq 0}$ is continuous in time (with respect to the operator norm) and Markovian. Note that the associated diffusion process, denoted (simply) by $X=X^{(\beta)}\df(X_t)_{t\geq 0}$ is called the squared Bessel process
of dimension $2\beta>0$ (up to a time scaling by a factor 2). It is the  solution
to the stochastic differential equation
\bqn{Bessel}
\fo t\geq 0,\qquad dX_t&=&\sqrt{2X_t}dB_t+\beta dt\eqn
where $B\df(B_t)_{t\geq 0}$ is a standard real Brownian motion.
The link between $\Bs^{(\beta)}$ and $X$ can be characterized, $\fo t\geq 0,\,\fo f\in\cC_{0}(\RR_+)$,  by
\bqn{Pb}
\fo x\in\RR_+\qquad \Bs^{(\beta)}_t f(x)&=&\EE_{x}[f(X_t)]\eqn
where we recall that $\cC_{0}(\RR_+)$ is the space of continuous functions  on $\RR_+$ vanishing at infinity and where
the $x$ in index of the expectation indicates that $X$ started with $X_0=x$.
For all these assertions, see for instance  Chapter XI of the book of Revuz and Yor \cite{MR1725357}.
\par
Next, consider $f\in\esD(\Bg_{\beta})$. Then, the mapping $\RR_+\ni t\mapsto \Bs^{(\beta)}_t f \in\esL^2(\mu_{\beta})$ is continuously differentiable and we have
\bq
\fo t\geq 0,\qquad \pa_t \Bs^{(\beta)}_t f\ =\ (\Bg_{\beta}\Bs^{(\beta)}_t)f\ =\ (\Bs^{(\beta)}_t\Bg_{\beta}) f\eq
(for any $f\in\esL^2(\mu_{\beta})$, this is true for positive $t>0$).
We equally have, for any $f\in \esD(\dBg_{\beta})$,
\bq
\fo t\geq 0,\qquad \pa_t \dBs^{(\beta)}_t f \ =\ (\dBg_{\beta}\dBs^{(\beta)}_t) f\ =\ (\dBs^{(\beta)}_t\dBg_{\beta}) f.\eq
\par
Fix $t\geq 0$, $f\in \esD(\dBg_{\beta})$ and consider the mapping $[0,t]\ni s\mapsto \Bs^{(\beta)}_s\Lambda \dBs^{(\beta)}_{t-s}f\in\esL^2(\mu_{\beta})$.
Taking into account that the three  operators in this expression are bounded by $1$ in norm, we get
\bq
\fo s\in[0,t],\qquad
\pa_s \Bs^{(\beta)}_s\Lambda \dBs^{(\beta)}_{t-s} f&=&\Bs^{(\beta)}_s\Bg_{\beta}\Lambda \dBs^{(\beta)}_{t-s} f-\Bs^{(\beta)}_s\Lambda \dBg_{\beta}\dBs^{(\beta)}_{t-s} f \\
&=&\Bs^{(\beta)}_s(\Bg_{\beta}\Lambda-\Lambda \dBg_{\beta}) \dBs^{(\beta)}_{t-s}f\\
&=&0\eq
due to \eqref{b3}.
The gateway relationship \eqref{eq:gatewayLag}  follows by integration in $s\in[0,t]$, at least on $\esD(\dBg_{\beta})$.
By density of
$\esD(\dBg_{\beta})$ in $\esLd^2(\fm_{\beta})$ and continuity of the operators $\Bs^{(\beta)}_t\Lambda$ and $\Lambda \dBs^{(\beta)}_t$, see Lemma \ref{lem:lambda_bounded}, the formula is extended, by a density argument, to $\esLd^2(\fm_{\beta})$.
The second formula is obtained by similar considerations, via the mapping
$[0,t]\ni s\mapsto \dBs^{(\beta)}_s\Lambda^* \Bs^{(\beta)}_{t-s}f$ for $f\in\esD(\Bg_{\beta})$, or by taking the adjoint relation in the first formula. Finally, since $\Lambda$ is a quasi-affinity between Hilbert spaces and the operators are self-adjoint, the fact that the gateway relationship can be lifted to an unitary equivalence is justified in  \cite[Lemma 4.1]{Douglas}.

\subsubsection{The Feller case}
We now prove the gateway identity of Theorem \ref{thm:main} in the Banach space $\esc_0(\ZZ_+)$.
On the one hand, from \cite[Chap. XI]{MR1725357}, we have, recalling that $\fe_{-\lambda}(x)=e^{-\lambda x}$, for any $\lambda,x, t\geq 0$,
\begin{equation*} \label{eq:lt_bessel}
\cont{Q}^{(\beta)}_t \fe_{-\lambda}(x)=\E_x\left[e^{-\lambda \cont{X}_t}\right]= (1+\lambda t)^{-\beta}e^{-x\frac{\lambda}{1+\lambda t}},
\end{equation*}
and, since for any $|s|<1$, $\mathfrak{p}_s\in \esc_0(\ZZ_+)$ and for any $x\in \R^+$,
\[\Lambda \mathfrak{p}_s(x) =\sum_{n\in\ZZ_+}\frac{(sx)^n}{n!}e^{-x}=e^{-(1-s)x}=\fe_{s-1}(x),\]
we get
\begin{equation*}
  \cont{Q}^{(\beta)}_t \Lambda\mathfrak{p}_s(x)=\left(1+(1-s)t\right)^{-\beta}
  \exp\left(-x\frac{1-s}{1+(1-s)t}\right)^{n}. 
\end{equation*}
On the other hand, using the Feyman-Kac formula, combined with the method of characteristic curves for solving the corresponding PDE, see e.g.~\cite[Chap.~4]{Dawson} for the case $\beta=0$ but the general case follows in a similar way, one gets, for any $t\geq 0$ and $|s|<1$,
\begin{equation*}
 \discret{Q}^{(\beta)}_t \mathfrak{p}_s(n)=\E_n\left[ s^{\discret{X}_t}\right]=\left(1+(1-s)t\right)^{-\beta}
  \left(\frac{1+(t-1)(1-s)}{1+(1-s)t}\right)^{n},
\end{equation*}
yielding
\begin{eqnarray*}
 \Lambda \discret{Q}^{(\beta)}_t \mathfrak{p}_s(x)&=&\left(1+(1-s)t\right)^{-\beta}
  \exp\left(-x\left(1-\frac{1+(t-1)(1-s)}{1+(1-s)t}\right)\right) \\
  &=&\left(1+(1-s)t\right)^{-\beta}
  \exp\left(-x\frac{1-s}{1+(1-s)t}\right)^{n}.
\end{eqnarray*}
Hence for any $|s|<1$,
\begin{eqnarray*}
 \Lambda \discret{Q}^{(\beta)}_t \mathfrak{p}_s(x)&=&\cont{Q}^{(\beta)}_t \Lambda\mathfrak{p}_s(x).
\end{eqnarray*}
We complete the proof by recalling that the linear span of $\{\mathfrak{p}_s,|s|<1\}$ is dense in $\esc_0(\ZZ_+)$ and by invoking the continuity of the involved linear operators, see Lemma \ref{lem:lambda_def}.

\section{Some consequences of the gateway relation  \eqref{eq:self-bessel}}\label{sec:cons-bes}
In this section, we provide the proof of  Propositions \ref{prop:self} and \ref{prop:beta-gam-d} and also present some additional applications of the gateway relationship between the squared Bessel semi-groups and the linear birth-death ones.
\subsection{Proof of Proposition \ref{prop:self}}
Let us recall the $d$-self-similarity property  enjoyed by the Bessel semi-group,  for any $\sigma,x,t>0$,
   \begin{equation}\label{eq:self-bessel_rec}
     \cont{Q}^{(\beta)}_{t}\cont{d}_{\sigma} f(x)=\E_x[f(\sigma X_t)]=\E_{\sigma x}[f( X_{\sigma t})]=\cont{d}_{\sigma}\cont{Q}^{(\beta)}_{\sigma t}f(x).
   \end{equation}
 We also recall that the family of linear operators $(\cont{d}_{e^{-t}})_{t\geq 0}$, where we recall that $\cont{d}_{e^{-t}}f(x)=f(e^{-t}x)$, form a group and corresponds to the (Markovian) dynamical system $\frac{d}{dt}{x(t)}=x(t)$. By means of the gateway relation \eqref{eq:self-bessel}, we  can also get  a \emph{discrete scaling property} for the birth-and-death process $\discret{X}$.
To this end, we introduce the binomial  kernel $\discret{D}_\sigma $ on $\ZZ_+$ given by
\bq
\fo n,m\in\ZZ_+,\qquad \discret{D}_\sigma (n,m)&\df&\binom{n}{m}\sigma^m(1-\sigma )^{n-m}\eq
and recall the notation $
 \discret{D}_\sigma f(n)=\sum_{m=0}^n f(m)\discret{D}_\sigma (n,m)$
which will play a role analogous to $d_\sigma$.
Note that it is Markovian only for $\sigma\in[0,1]$.  The first interest  of $\discret{D}_\sigma $  comes
from the following intertwining relation, that specifies the gateway relation \eqref{eq:dil-D}.
\begin{lem}\label{lem:dLamD}
We have, for any $\sigma>0$, on $\esc_0(\ZZ_+)$,
\bq
\cont{d}_{\sigma}\Lambda&=&\Lambda \cD_{\sigma}.\eq
Moreover $ (\cD_{e^{-t}})_{t \geq 0}$ is the semi-group of the dual Yule process, a pure-death process.
\end{lem}
\begin{rem}
Note that in \cite{Biane-int}, Biane, resorting to a  group theoretic approach, derives the following intertwining relation
\[\Gate{\cD_{e^{-t}}}{H}{U_t}\]
where $U=(U_t)_{t\geq0}$ is the semi-group of the classical Ornstein-Uhlenbeck on $\RR$ and
$H f(n) = \sqrt{\frac{2}{\pi}}\frac{1}{n!}\int_{\R}f(x)h^2_n(x)e^{-2x^2}dx$ is, with $h_n$ the Hermite polynomial,  a Markov kernel.
\end{rem}
\begin{rem}
Observe that, despite the fact that $\discret{D}_\sigma $ is not Markovian for $\sigma>1$, the operator $\Lambda \discret{D}_\sigma =d_\sigma\Lambda$ is always Markovian.
\end{rem}
\proof
Let $f\in \esc_0(\ZZ_+)$ be a test function. Then, for any $\sigma>0$, we have
\bq
d_\sigma\Lambda f(x)&=&e^{-\sigma x}\sum_{m\in\ZZ_+}f(m)\frac{(\sigma x)^m}{m!}\\
&=&e^{-x}\sum_{m\in\ZZ_+}f(m)\frac{(\sigma x)^m}{m!}\exp((1-\sigma )x)\\
&=&e^{-x}\sum_{m\in\ZZ_+}f(m)\frac{(\sigma x)^m}{m!}\sum_{n\geq m}\frac{1}{n-m}((1-\sigma )x)^{n-m}\\
&=&e^{-x}\sum_{n\in\ZZ_+}\sum_{m=0}^n\binom{n}{m}\sigma^m(1-\sigma )^{n-m}f(m)\frac{x^n}{n!}\\
&=&e^{-x}\sum_{n\in\ZZ_+}\discret{D}_\sigma  f(n)\frac{x^n}{n!}\\
&=&\Lambda \discret{D}_\sigma  f(x).
\eq
The fact that $ (\cD_{e^{-t}})_{t \geq 0}$ is the semi-group of a pure-death process is well-known and can be found in \cite[Proposition 3.3]{Biane-int}.
\wwtbp
We proceed with the proof of the discrete scaling property, stated in \eqref{eq:self-besseld}, for the semi-group $\dBs^{(\beta)}$ of the birth-and-death process, which is analogous to \eqref{eq:self-bessel_rec}. First,
multiply  \eqref{eq:self-bessel_rec}, the intertwining of the squared-Bessel semi-groups with $d_\sigma$, on the right by $\Lambda$, to get, on $\esc_0(\ZZ_+)$,
\bqn{dcPL}d_\sigma\Bs^{(\beta)}_{\sigma t}\Lambda &=&\Bs_t^{(\beta)} d_\sigma\Lambda.\eqn
By means of  the gateway relation \eqref{eq:gatewayLag} and the commutation relation of Lemma \ref{lem:dLamD},  the left-hand side can be written as
\bq
d_\sigma\Bs_{\sigma t}^{(\beta)}\Lambda\ =\ d_\sigma \Lambda \dBs_{\sigma t}^{(\beta)}\ =\ \Lambda \discret{D}_\sigma  \dBs_{\sigma t}^{(\beta)}\eq
whereas the right-hand side of \eqref{dcPL} is equal, using the same relations in a reverse order, to
\bq
\Bs_t^{(\beta)} d_\sigma\Lambda\ =\ \Bs_t^{(\beta)}\Lambda \discret{D}_\sigma \ =\ \Lambda \dBs_t^{(\beta)} \discret{D}_\sigma . \eq
The announced result is now a consequence of the equality $\Lambda( \discret{D}_\sigma  \dBs_{\sigma t}^{(\beta)}-\dBs_t^{(\beta)} \discret{D}_\sigma )=0$
and of the injectivity property of $\Lambda$ obtained in Lemma \ref{Linj}.\par
\wwtbp

\subsection{Proof of Proposition \ref{prop:beta-gam-d}}
We start with the following lemma.
\begin{lem}\label{lem:BLb}
For any $\alpha,\beta>0$, we have on $\esc_0(\ZZ_+)$,
\begin{equation}\label{eq:gatBLb}
  \Gate{B_{\beta,\alpha}}{\Lambda}{\discret{B}_{\beta,\alpha}}
\end{equation}
  where $\discret{B}_{\beta,\alpha}: \esc_0(\ZZ_+)\mapsto \esc_0(\ZZ_+)$ is the Markov kernel defined, for any $n\in \ZZ_+$, by
   \[   \discret{B}_{\beta,\alpha}g(n)=(B_{\beta,\alpha} \odot  n) g=\sum_{m=0}^n g(m)\binom{n}{m}\E[B_{\beta,\alpha}^m(1-B_{\beta,\alpha})^{n-m}].\]
\end{lem}
\begin{proof}
Let $g$ be a test function in $\esc_0(\ZZ_+)$ that  we choose, without lose of generality, to be non-negative. Then, one has, for any  $x>0$, that
\begin{eqnarray*}
B_{\beta,\alpha} \Lambda g(x)= \E\left[\Lambda g(x {\rm{B}}(\alpha,\beta))\right] &=& \int_{0}^{1} d_\sigma \Lambda g(x) \P({\rm{B}}(\alpha,\beta))\in d\sigma) \\
&=&  \int_{0}^{1} \Lambda \discret{D}_{\sigma} g(x) \P({\rm{B}}(\alpha,\beta))\in d\sigma)\\
&=& \Lambda \int_{0}^{1}  \discret{D}_{\sigma} g(.) \P({\rm{B}}(\alpha,\beta))\in d\sigma)(x)
 \end{eqnarray*}
where we used for the third identity Lemma \ref{lem:dLamD}.
We complete the proof of the lemma by observing that for any $n\in \ZZ_+$, \[\int_{0}^{1}  \discret{D}_{\sigma} g(n) \P({\rm{B}}(\alpha,\beta))\in d\sigma)=\sum_{m=0}^n g(m)\binom{n}{m}\E[B_{\beta,\alpha}^m(1-B_{\beta,\alpha})^{n-m}].\]
\end{proof}

Next, recalling from Carmona et al.~\cite{Carmona-Petit-Yor-98} that  for any $\alpha,\beta>0$, on $\esC_0(\RR_+)$,
\[ \Gate{Q^{(\alpha+\beta)}_t}{ {{B}}_{\beta,\alpha}}{Q^{(\beta)}_t}.\]
Multiplying both sides by $\Lambda: \esc_0(\ZZ_+) \mapsto \esC_0(\RR_+)$ to the right, we obtain, on $\esc_0(\ZZ_+)$,
\[  Q^{(\alpha+\beta)}_t{{B}}_{\beta,\alpha}\Lambda = {{B}}_{\beta,\alpha} Q^{(\beta)}_t\Lambda.\]
Then,  Lemma \ref{lem:BLb} and the gateway  relation \eqref{eq:self-bessel} yield
\[  \Lambda \discret{Q}^{(\alpha+\beta)}_t{\discret{B}}_{\beta,\alpha} = \Lambda {\discret{B}}_{\beta,\alpha} \discret Q^{(\beta)}_t\]
which completes the proof of the intertwining relation \eqref{eq:int-dbeta-gamma} by invoking the injectivity of $\Lambda$ on $\esc_0(\ZZ_+)$, see Lemma \ref{lem:lambda_bounded}.
Then, since ${\discret{B}}_{\beta,\alpha} g(0)=g(0)$ and
\[   \discret{Q}^{(\alpha+\beta)}_1{\discret{B}}_{\beta,\alpha}g(0) =  {\discret{B}}_{\beta,\alpha} \discret Q^{(\beta)}_1 g(0),\]
we get  that
\[   B_{\beta,\alpha} \odot{\rm{Pois}}({\rm{Gam}}(\alpha+\beta))\stackrel{(d)}{=}{\rm{Pois}}({\rm{Gam}}(\beta))\]
which is the sought identity.

\subsection{The time-inversion property}
   Another interesting symmetry of the Bessel semi-group is the time-inversion property which  says that, for any $t>0$,
   \begin{equation}\label{eq:tinv-bessel}
    \cont{Q}^{(\beta)}_{\frac{1}{t}}\cont{d}_{t^2}f(0)=\cont{Q}^{(\beta)}_tf(0),
   \end{equation}
   which has its discrete counterpart.
    \begin{proposition}\label{prop:tiv}
    For any $t>0$, we have for any bounded or non-negative function $g$ on $\ZZ_+$,
   \begin{equation*}
    \discret{Q}^{(\beta)}_{\frac{1}{t}}\discret{D}_{t^2}g(0)=\discret{Q}^{(\beta)}_tg(0)=\E[g({\rm{Pois}}(t{\rm{Gam}}(\beta)))].
   \end{equation*}
    \end{proposition}
\begin{proof}
  Using successively the gateway relation \eqref{eq:gatewayLag} and  the time-inversion property of the Bessel \eqref{eq:tinv-bessel}, we obtain that, for any $t>0$,
  \begin{eqnarray*}
    \Lambda \discret{Q}_t g(0) &=& \cont{Q}^{(\beta)}_t \Lambda g(0) \\
     &=& \cont{Q}^{(\beta)}_{\frac{1}{t}}\cont{d}_{t^2}\Lambda g(0)\\
     &=& \cont{Q}^{(\beta)}_{\frac{1}{t}}\Lambda \discret{D}_{t^2}g(0)\\
     &=& \Lambda\discret{Q}^{(\beta)}_{\frac{1}{t}} \discret{D}_{t^2}g(0)
      \end{eqnarray*}
   where for the third identity we used Proposition \ref{prop:comut-Dd}. To complete the proof of the first identity, we observe that $\Lambda g(0)=g(0)$. Finally, using this last identity,  the  gateway relation \eqref{eq:gatewayLag} and the $d$-self-similarity of $Q^{(\beta)}$, one deduces that
   \[ \discret{Q}^{(\beta)}_tg(0)= Q^{(\beta)}_t \Lambda g(0)= \cont{d}_t Q^{(\beta)}_t \Lambda g(0) = Q^{(\beta)}_t \cont{d}_t\Lambda g(0) =\E\left[g({\rm{Pois}}(t{\rm{Gam}}(\beta)))\right].\]
\end{proof}

\subsection{The Laguerre polynomials as Jensen's polynomial of the Bessel functions}
Since, for any  $\beta >0$, its infinitesimal generator is self-adjoint in the  Hilbert space $\esL^{2}({\mu}_{\beta})$, see  the proof of Lemma \ref{lem:gate_gen_B},  $ (\cont{Q}^{(\beta)}_t)_{t\geq0}$ is,  a self-adjoint contraction semi-group in  $\esL^{2}({\mu}_{\beta})$, where ${\mu}_{\beta}$ is its speed measure which, we recall, is
${\mu}_{\beta}(dx)=\frac{x^{\beta-1}}{\Gamma(\beta)}dx, \: x>0.$
Next, we write, for $z \in \mathbb{C}$,
 \begin{equation*} \label{eq:def_bessel}
\Jg(z)=\Gamma(1+\beta)\sum_{n=0}^{\infty}\frac{(e^{i\pi}z)^n}{n!\Gamma(n+1+\beta)}=\Gamma(1+\beta)z^{-\frac{\beta}{2}}J_{\beta}(2\sqrt{z})
\end{equation*}
 where $J_{\beta}$ denotes the usual Bessel function of order $\beta$ and we named $\Jg$ the
  normalized Bessel function as $\Jg(0)=1$.
 Then, we define the Hankel transform of order $\beta$ of a function $f\in \Lv$ by
\begin{equation*} \label{eq:defn_Hankel}
\Hg f(q) = \int_0^{\infty} \Jg(qx)f(x)\mu_{\beta}(dx), \quad q>0,
\end{equation*}
where the integral is understood in the $\esL^{2}$-sense as $\Jg \notin \Lv$.
Then, $\frac{1}{\Gamma(1+\beta)}\Hg$ is a self-reciprocal isometry of $\Lv$. 
Moreover, we have for any $t>0$, the following diagonalization in $\Lv$
 of the transition densities of $Q^{(\beta)}_t$ with respect to the reference measure $\mu_{\beta}$, \begin{equation*} \label{eq:diagP}
Q^{(\beta)}_t (x,y)=\Hg e_{t}\cont{d}_y\Jg (x)=\int_0^{\infty}e^{-qt}\Jg(qy) \Jg(qx)\mu_{\beta}(dq).
\end{equation*}
For more details, see for example \cite{Muck-Stein}.
Note that for any $x\geq 0, q>0$, $\cont{d}_q\Jg(x)=\Jg(qx)$ is solution to
\begin{equation*}\label{eq:bessel-eigen}
  \cont{G}_{\beta}\: \cont{d}_q\Jg(x)=q\: \cont{d}_q\Jg(x)
\end{equation*}
where $\cont{G}_{\beta}$ is here the differential operator not the generator of $\cont{Q}_t^{(\beta)}$ as $\Jg \notin \Lv$. It means  that  $\cont{Q}_t^{(\beta)}$ has a continuous spectrum given by  $(e^{-qt})_{q\in \R^+}$.
Similarly, from Karlin and McGregor \cite{KMcGlinear}
we have that
%
for any $t>0$,  the following diagonalization  of the transition kernel of  $\mathbb Q^{(\beta)}_t$ in $\Lpv$
\begin{equation*} \label{eq:spectral_expansion_Q_t}
\mathbb Q^{(\beta)}_t (n,m)=\int_0^{\infty}e^{-qt} {\mathbb{F}}^{(\beta)}_q(m) {\mathbb{F}}^{(\beta)}_q(n)\mu_{\beta}(dq).
\end{equation*}
where  $(\mathcal{L}^{(\beta)}_n(q)={\mathbb{F}}^{(\beta)}_q(n))_{n\geq 0}$ stands for the Laguerre polynomials that are defined as
\begin{equation*} 
\mathcal{L}^{(\beta)}_k(q)=\sum_{r=0}^k (-1)^r { k +\beta \choose k-r}  \frac{q^r }{r!}.
\end{equation*}
Note that this expansion could also be derived from the diagonalization of the Bessel semi-group and the gateway relationship \eqref{eq:gatewayLag}. However, we postpone to Section \ref{sec:spect_lag} for the application of intertwining relationship for the spectral decomposition of Markov semi-group. In this vein, we refer the interested readers to the papers  \cite{Patie-Savov-16}, \cite{Choi_similarity} and \cite{Choi_skip} where a methodology based on this concept is established to study the spectral theory of  non-reversible Markov semi-groups. We are ready to state the following.
\begin{proposition}
Let $q>0$. Then, we have, for all $n \in \ZZ_+$,
\begin{equation}\label{eq:LJL}
    \adjL  \cont{d}_q \Jg(n)= e^{-q}{\mathbb{F}}^{(\beta)}_q(n),
 \end{equation}
and, for all $x\geq 0$,
  \begin{equation} \label{eq:Jensen} \Lambda \mathbb{F}^{(\beta)}_q(x)=\cont{d}_q\Jg(x). \end{equation} 
\end{proposition}
and thus $\mathbb{F}^{(\beta)}_q(x)=\discret{D}_q \mathbb{F}^{(\beta)}_1(x)$.
\begin{remark}
  The non-Markovian transform  $f\mapsto e^{x}\Lambda \discret{D}_qf(x)$  is known as the Jensen's transform in the special function literature  and it associates  polynomials (the Jensen polynomials) to  entire functions.  It has the interesting feature to  preserve the reality of zeros, see \cite{Jensen}. In our context, it is well-known that the Laguerre polynomials are the Jensen polynomials of the Bessel function and both have only positive real zeros.
\end{remark}
\begin{proof}
First, we have for any $n\in \N$ and $q>0$,
\begin{eqnarray*}
  \Lambda^* d_q \Jg(n) &=& \E\left[\Jg(q G(n+\beta+1))\right]\\&=&
  \sum_{k=0}^{\infty}\int_0^{\infty}\Gamma(1+\beta)\frac{(e^{i\pi}qx)^k}{k!\Gamma(k+1+\beta)}e^{-x}x^{n+\beta}\frac{dx}{\Gamma(n+\beta+1)}\\
  &=&\frac{\Gamma(1+\beta)}{\Gamma(n+\beta+1)} \sum_{k=0}^{\infty}\frac{\Gamma(k+n+1+\beta)}{k!\Gamma(k+1+\beta)}(e^{i\pi}q)^k \\
  &=& {}_{1}F_{1}(n+1+\beta,1+\beta,-q)\\ &=& e^{-q}\: {}_{1}F_{1}(-n,1+\beta,q)\\ &=& \mathbb{F}^{(\beta)}_q(n)
\end{eqnarray*}
where the interchange of the sum signs is justified by a classical Fubini argument, see \cite{Titchmarsh39}, $ {}_{1}F_{1}$ stands for the Kummer function and   the last sequence of identities follow from classical properties of the hypergeometric function, see e.g.~\cite{Koekoek2010}. Though the identity \eqref{eq:Jensen} is well-known, see e.g.~\cite[Proposition 2.1(ii)]{Craven1989}, the last relation can be easily deduce from  this latter as, for any $q,x>0$, $\Lambda \mathbb{F}^{(\beta)}_q(x)=\cont{d}_q\Jg(x)=\cont{d}_q\Lambda \mathbb{F}^{(\beta)}_1(x)=\Lambda \discret{D}_q\mathbb{F}^{(\beta)}_1(x)$ where for the last equality we used Lemma \ref{lem:dLamD}.
 \end{proof}

\subsection{Products of the intertwining kernels}
First, note that  the identities \eqref{b3} and \eqref{dd2} yield on $\esC_0(\RR_+)$
\bq
\Bg_{\beta}\Lambda\adjL  \ =\ \Lambda \dBg_{\beta} \adjL \ =\ \Lambda\adjL  \Bg_{\beta} \eq
and similarly
\bq
\fo t\geq 0,\qquad
\Bs^{(\beta)}_t\Lambda\adjL &=&\Lambda\adjL \Bs_t^{(\beta)}\eq
(more generally, we can expect that $F(\Bg_{\beta})\Lambda\adjL =\Lambda\adjL F(\dBg_{\beta})$
for any measurable function $F\st (-\iy,0]\ri\RR$, via functional calculus and the appropriate inclusion of the domains). Thus it appears
that the operator $\Lambda\adjL \st \esL^2(\mu_{\beta})\ri\esL^2(\mu_{\beta})$ commutes with the whole semi-group
$\Bs^{(\beta)}$. One can go further and compute $\Lambda\adjL $ as follows.
\begin{pro}\label{LL*}
We have
\bq
\Lambda\adjL &=&\Bs^{(\beta)}_1.
\eq
\end{pro}
This formula may look strange at first view since $\beta$ does not appear explicitly in the left-hand side,
but $\beta$ is hidden in the definition of $\adjL $, which depends on the spaces $\esL^2(\mu_{\beta})$ and $\esLd^2(\fm_{\beta})$.
\proof
Consider a non-negative and measurable mapping $f\st \RR_+\ri \RR_+$. By definition, we have for any $x>0$,
\bq
\Lambda\adjL  f(x)&=&
\sum_{n\in\ZZ_+} \adjL f(n) \frac{x^n}{n!}e^{-x}\\
&=&\sum_{n\in\ZZ_+} \int_0^{\iy}f(y)\frac{y^{n+\beta-1}}{\Gamma(n+\beta)}e^{-y}\, dy \frac{x^n}{n!}e^{-x}
\\
&=&
\int_0^{\iy}f(y)\lt(\frac{y}{x}\rt)^{(\beta-1)/2}e^{-y-x}\sum_{n\in\ZZ_+} \frac{(xy)^{n+(\beta-1)/2}}{\Gamma(n+\beta)n!}\, dy.
\eq
We recognize that
\bq
\sum_{n\in\ZZ_+}  \frac{(xy)^{n+(\beta-1)/2}}{\Gamma(n+\beta)n!}&=&I_{\beta-1}(2\sqrt{xy})\eq
where $I_{\beta-1}$ is the modified Bessel function of the first kind of index $\beta-1$.
From Dufresne \cite{MR2181585} (take $t=1/2$ there due to our time scaling, see also Corollary 1.4 of Chapter XI  of Revuz and Yor \cite{MR1725357},
but a factor $1/t$ is missing in their formula), we get that  the measure, on $\RR_+$, $
\lt(\frac{y}{x}\rt)^{(\beta-1)/2}e^{-y-x}I_{\beta-1}(2\sqrt{xy})\un_{\RR_+}(y) dy$ is the law of $X^{(\beta)}_1$ under $\PP_x$,
namely we have, for any $x>0$,
\bq
\Lambda\adjL f(x)&=&\Bs^{(\beta)}_1 f (x).\eq
This relation is also true for $x=0$. Indeed the Poisson law of parameter 0 is just the Dirac mass in 0, so that
\bq
\Lambda\adjL f(0)&=& \adjL f(0)=\int_0^{\iy}f(y)\frac{y^{\beta-1}}{\Gamma(\beta)}e^{-y}\, dy \eq
and according to Corollary 1.4 of Chapter XI  of Revuz and Yor \cite{MR1725357}, the measure
$\frac{y^{\beta-1}}{\Gamma(\beta)}e^{-y}\, dy$ is the entrance law at time $1$ of the Bessel process of dimension $2\beta$ starting from $0$.
\par
\noindent In summary, we have proven the commutative diagram displayed in Figure \ref{fig1} and valid for any $\beta>0$ and $t>0$.
\begin{figure}[!h]\centering
\begin{tikzcd}[scale=2]
\esL^2(\mu_\beta) \arrow[r, "\Bs^{(\beta)}_t"] \arrow[d, "\Lambda"'] \arrow[dd, bend right=50 , "\Bs_1^{(\beta)}"']
& \esL^2(\mu_\beta) \arrow[d, "\Lambda" ]  \arrow[dd, bend left=50 , "\Bs_1^{(\beta)}"]\\
\esLd^2(\fm_{\beta}) \arrow[r, "\dBs^{(\beta)}_t" ]\arrow[d, "\adjL "']
& \esLd^2(\fm_{\beta}) \arrow[d, "\adjL " ]\\
\esL^2(\mu_\beta) \arrow[r, "\Bs^{(\beta)}_t"]
& \esL^2(\mu_\beta)
\end{tikzcd}
\caption{Intertwining relations with $\Lambda\adjL =\Bs^{(\beta)}_1$}\label{fig1}
\end{figure}

In view of Proposition \ref{LL*}, it is natural to wonder what is $\adjL \Lambda$. 
\begin{pro}\label{L*L}
We have
\bq
 \adjL \Lambda&=&\dBs^{(\beta)}_1
\eq
and it follows that
\bq
\fo m,n\in\ZZ_+,\qquad
\dBs^{(\beta)}_1(n,m)&=&2^{-(m+n+\beta)} \frac{(m+n+\beta-1)(m+n+\beta-2)\cdots (n+\beta)}{m!}.
\eq
\end{pro}
\proof
Denote $R\df \adjL \Lambda$. From Proposition \ref{LL*}, we get
that $\Lambda R=\Lambda \adjL \Lambda =\Bs_1^{(\beta)}\Lambda=\Lambda \dBs_1^{(\beta)}$, namely
\bq
\Lambda (R-\dBs^{(\beta)}_1)&=&0.\eq
Lemma \ref{Linj} implies that $R=\dBs^{(\beta)}_1$. As a consequence, for any non-negative measurable function $f$ and $n\in\ZZ_+$,
we have
\bq
\dBs^{(\beta)}_1 f(n)&=&
\adjL \Lambda f(n)\\&=&
\frac1{\Gamma(n+\beta)}\int_0^{\iy} \Lambda f (x)\, x^{n+\beta-1}e^{-x}dx\\
&=&\frac1{\Gamma(n+\beta)}\int_0^{\iy} \sum_{l\in\ZZ_+}\frac{f(l)}{l!}x^l e^{-x}\, x^{n+\beta-1}e^{-x}dx\\
&=&\sum_{k\in\ZZ_+}\frac{f(k)}{k!\Gamma(n+\beta)}\int_0^{\iy} x^{k+n+\beta-1}e^{-2x}\,dx\\
&=&\sum_{k\in\ZZ_+}\frac{f(k)}{k!\Gamma(n+\beta)}\frac{\Gamma(k+n+\beta)}{2^{k+n+\beta}}\\
&=&\sum_{k\in\ZZ_+} f(k)2^{-(k+n+\beta)} \frac{(k+n+\beta-1)(k+n+\beta-2)\cdots (n+\beta)}{k!}\eq
and we end up with the announced result by replacing $f$ by the indicator function of $m\in\ZZ_+$.
\wwtbp

We deduce the  commutative diagram displayed in Figure \ref{fig2} which is valid for any $\beta>0$ and $t>0$ and is analogous to Figure \ref{fig1}.
\begin{figure}[!h]\centering
\begin{tikzcd}[scale=2]
\esLd^2(\fm_{\beta}) \arrow[r, "\dBs^{(\beta)}_t"] \arrow[d, "\adjL "'] \arrow[dd, bend right=50 , "\dBs_1^{(\beta)}"']
& \esLd^2(\fm_{\beta}) \arrow[d, "\adjL " ]  \arrow[dd, bend left=50 , "\dBs_1^{(\beta)}"]\\
\esL^2(\mu_\beta) \arrow[r, "\Bs^{(\beta)}_t" ]\arrow[d, "\Lambda"']
& \esL^2(\mu_\beta) \arrow[d, "\Lambda" ]\\
\esLd^2(\fm_{\beta}) \arrow[r, "\dBs^{(\beta)}_t"]
& \esLd^2(\fm_{\beta})
\end{tikzcd}
\caption{Intertwining relations with $\adjL \Lambda=\dBs^{(\beta)}_1$}\label{fig2}
\end{figure}
\par
In view of Propositions \ref{LL*} and \ref{L*L}, one can be left wondering about the role of the time 1 in $\Bs^{(\beta)}_1$ and $\dBs^{(\beta)}_1$.
Let us show how it is possible to replace 1 by any time $t>0$, by taking into account the scaling property of the Bessel processes $X^{(\beta)}\df (X_t^{(\beta)})_{t\geq 0}$.
More precisely, for any $\sigma>0$ and $x\geq 0$, the law of $(X^{(\beta)}_{\sigma t})_{t\geq 0}$ starting from $\sigma x$ is equal to the law of
$(\sigma X^{(\beta)}_t)_{t\geq 0}$, where $X^{(\beta)}$ is starting from $x$.
At the level of the semi-group $\Bs^{(\beta)}$, we recall via \eqref{Pb} that
\bqn{dcP}
\fo \sigma>0,\,t\geq 0,\qquad
d_\sigma\Bs_{\sigma t}^{(\beta)}&=&\Bs_t^{(\beta)}d_\sigma\eqn
where $d_\sigma$ is the dilatation operator acting on
$\esL^2(\mu_{\beta})$ via
\bq
 d_\sigma f (x)&=&f(\sigma x).\eq
Formula \eqref{dcP} holds on $\esL^2(\mu_{\beta})$, since the operator $\sigma^{b/2}d_\sigma$ is an isometry of $\esL^2(\mu_{\beta})$. Indeed performing a change of variables formula, it appears that for $f\in
\esL^2(\mu_{\beta})$,
\bq
\int_0^{\iy} (d_\sigma f(x))^2\,\mu_{\beta}(dx)&=&\int_0^{\iy} f^2(\sigma x)\frac{x^{\beta-1}}{\Gamma(\beta)}\,dx
=\frac{1}{\sigma^{\beta}}\int_0^{\iy} f^2(x)\frac{x^{\beta-1}}{\Gamma(\beta)}\,dx\\ &=&\sigma^{-\beta}\int_0^{\iy} f^2(x)\,\mu_{\beta}(dx).\eq
From this isometry property, we deduce that $(\sigma^{\beta/2} d_\sigma)^*=(\sigma^{\beta/2}d_\sigma)^{-1}$,
i.e.
\bqn{dc*}
\cont{d}^*_{\sigma}\ =\ \sigma^{-\beta/2}(\sigma^{\beta/2}d_\sigma)^{-1}\ =\  \sigma^{-\beta}\cont{d}_{1/\sigma}.\eqn
Formula \eqref{dcP} can also be interpreted as a composition of Markov kernels,
by seeing $d_\sigma$ as the transition kernel
\bq
\fo x,x'\in \RR_+,\qquad d_\sigma(x,dx')&=&\delta_{\sigma x}(dx').\eq
This is an instance where the dual of a Markov process is not Markovian since according to \eqref{dc*} the weight of $d_\sigma^*$ is $\sigma^{-\beta}$. Next, define
\bq
\Lambda_{\sigma}&\df& d_\sigma \Lambda\eq
which is a Markov kernel from $\RR_+$ to $\ZZ_+$.
Due to the above observation, $\Lambda_\sigma^*=\Lambda^* d_\sigma^*$ is not a Markov kernel, so  consider instead the  Markov kernel given by
\bq
\wi \Lambda _{\sigma}&\df& \sigma^{\beta}\Lambda_\sigma^*\ =\ \Lambda^*  \cont{d}_{1/\sigma}.\eq
Here is the analogue for the first column of Figure \ref{fig1}.
\begin{lem}\label{LcLc1}
For any $\sigma>0$, we have $\Lambda_{\sigma}\wi\Lambda_{\sigma}= \Bs_{1/\sigma}^{(\beta)}$.
\end{lem}
\proof
We observe that
\bq
\Lambda_{\sigma}\wi\Lambda_{\sigma}&= & d_\sigma \Lambda\adjL  \cont{d}_{1/\sigma}=d_\sigma\Bs_1^{(\beta)}\cont{d}_{1/\sigma}=d_\sigma\cont{d}_{1/\sigma}\Bs^{(\beta)}_{1/\sigma}=\Bs^{(\beta)}_{1/\sigma}
\eq
where we have taken into account \eqref{dcP} with $t=1/\sigma$.
\wwtbp

The following result completes the analogue of Figure \ref{fig1}.
\begin{lem}\label{cinter}
We have for any $\beta,\sigma>0$ and $t\geq 0$,
\bq
\Bs_t^{(\beta)}\Lambda_{\sigma}&=&\Lambda_{\sigma} \dBs_{\sigma t}^{(\beta)} \textrm{ on $\esc_0(\ZZ_+) \cup \ell^2({\fm_{\beta}})$}\\
\dBs_{\sigma t}^{(\beta)}\wi\Lambda_{\sigma}&=&\wi\Lambda_{\sigma} \Bs_{t}^{(\beta)} \textrm{ on $\esC_0(\RR_+) \cup \esL^2(\mu_\beta)$}.
\eq
\end{lem}
\proof
First, we observe that
\bq
\Bs_t^{(\beta)}\Lambda_{\sigma}&=&\Bs_t^{(\beta)} d_\sigma \Lambda=d_\sigma \Bs_{\sigma t}^{(\beta)} \Lambda=d_\sigma \Lambda \dBs_{\sigma t}^{(\beta)} =\Lambda_{\sigma} \dBs_{\sigma t}^{(\beta)}\eq
where we used the scaling property \eqref{eq:self-bessel_rec} and the main gateway relationship \eqref{eq:gatewayLag} which both hold  on $\esc_0(\ZZ_+)$. The extension to  $\ell^2({\fm_{\beta}})$  is obtained by a standard density argument. On the other hand, on $\esC_0(\RR_+)$,
\bq
\dBs_{\sigma t}^{(\beta)}\wi\Lambda_{\sigma}&=&\sigma^b\dBs_{\sigma t}^{(\beta)}\adjL  \cont{d}_{1/\sigma}=\sigma^b\adjL \Bs_{\sigma t}^{(\beta)} \cont{d}_{1/\sigma}=\sigma^b\adjL \cont{d}_{1/\sigma} \Bs_{t}^{(\beta)}=\wi\Lambda_{\sigma} \Bs_{t}^{(\beta)}
\eq where we used
\eqref{dcP} in the third equality, which by resorting, again, to a density argument completes the proof.
\par
\wwtbp
To get the analogue of Figure \ref{fig2}, it is sufficient to adapt the first relation of
Proposition \ref{L*L}.
\begin{lem}\label{LcLc}
For any $\beta,\sigma>0$, we have
\bq\wi\Lambda_{\sigma}\Lambda_{\sigma}&=&\dBs^{(\beta)}_{1}.\eq
\end{lem}
\proof
Set $R_{\sigma}\df \wi\Lambda_{\sigma}\Lambda_{\sigma}$, then we have
$\Lambda_{\sigma}R_{\sigma}=\Lambda_{\sigma}\wi\Lambda_{\sigma}\Lambda_{\sigma}=\Bs_{1/\sigma}^{(\beta)}\Lambda_{\sigma}=\Lambda_{\sigma} \dBs_{1}^{(\beta)}$, namely
\bq
d_\sigma\Lambda(R_{\sigma}-\dBs_{1}^{(\beta)})&=&0.\eq
We get the announced result, since $\sigma^{\beta/2}d_\sigma$ is an isometry and $\Lambda$ is injective, see Lemma \ref{lem:lambda_bounded}.
\wwtbp

To sum up, we have proven the commutative diagram of Figure \ref{fig3} valid  for any $\beta,t,s>0$ (by taking $\sigma=1/s$ in the above considerations).
\begin{figure}[!h]\centering
\begin{tikzcd}[scale=2]
\esL^2(\mu_\beta) \arrow[r, "\Bs^{(\beta)}_t"] \arrow[d,color=red, "\Lambda_{1/s}"] \arrow[dd, bend right=50 , color=blue, "\Bs_s^{(\beta)}"']
& \esL^2(\mu_\beta) \arrow[d, "\Lambda_{1/s}"' ]  \arrow[dd, bend left=50 , "\Bs_s^{(\beta)}"]\\
\textcolor{red}{\esLd^2(\fm_{\beta})} \arrow[r, color=red,"\dBs^{(\beta)}_{t/s}" ]\arrow[d,  "\wi\Lambda_{1/s}"]\arrow[dd, bend right=50 , "\dBs_1^{(\beta)}"']
&\textcolor{red}{\esLd^2(\fm_{\beta})} \arrow[d, color=red,"\wi\Lambda_{1/s}"' ]\arrow[dd, bend left=50 , "\dBs_1^{(\beta)}"]\\
\textcolor{blue}{\esL^2(\mu_\beta)} \arrow[r, color=blue, "\Bs^{(\beta)}_t"] \arrow[d, "\Lambda_{1/s}" ]
& \esL^2(\mu_\beta) \arrow[d, "\Lambda_{1/s}"' ]\\
\esLd^2(\fm_{\beta}) \arrow[r, "\dBs^{(\beta)}_{t/s}"]
& \esLd^2(\fm_{\beta})
\end{tikzcd}
\caption{Intertwining relations with $\Lambda_{1/s}\wi\Lambda_{1/s}=\Bs^{(\beta)}_s$}\label{fig3}
\end{figure}
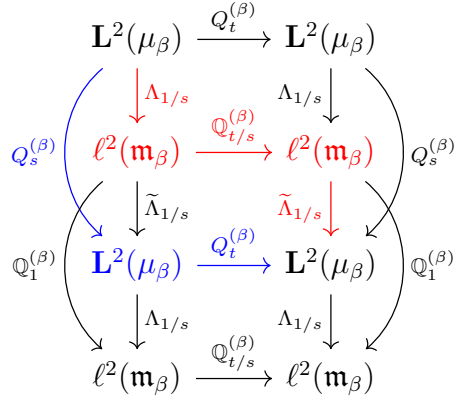
\par\me

\subsection{Exact simulation of Bessel processes}
For given $x\in\RR_+$ and $t>0$, assume that we want to sample $X^{(\beta)}_t$ under $\PP_x$.
There are two traditional ways to do it:
\begin{itemize}
\item Solve the stochastic differential equation  \eqref{Bessel}. In practice it can be done via Euler schemes, preferentially implicit ones, due to the fact that the diffusion term
$\sqrt{2X^{(\beta)}_t}$ can be quite big and to avoid that the approximation crosses 0.
\item Use the formula giving the density of law of $X^{(\beta)}_t$ under $\PP_x$, see Corollary 1.4 of Chapter XI  of Revuz and Yor \cite{MR1725357}.
Note that the Bessel function of index $\beta-1$ enters in this formula.
\end{itemize}
So both these solutions require some approximations and do not provide an exact sampling.
\par\sm
Let us show how the intertwining relations \eqref{b3} and \eqref{dd2} can be used to provide a simple exact sampling of $X^{(\beta)}_t$  under $\PP_x$.
We will first do it for $t\geq 1$, but due to the scaling property of $X^{(\beta)}$, the construction will next be extended to any $t>0$.
\par
\sm
Let $\dBs^{(\beta)}\df(\dBs^{(\beta)}_t)_{t\geq 0}$ be the Markov semi-group generated by $\dBg_{\beta}$. 
It is very simple to simulate an associated birth and death process $\cX^{(\beta)}\df(\cX_t^{(\beta)})_{t\geq 0}$:
first choose $\cX_0^{(\beta)}$ according to a given initial distribution.
Then sample an independent exponential time $\tau_1$ of parameter $\vert \dBg_{\beta}(\cX^{(\beta)}_0,\cX^{(\beta)}_0)\vert$.
For $t\in[0,\tau_1)$, we take $\cX^{(\beta)}_t\df \cX^{(\beta)}_0$.
Choose $\cX^{(\beta)}_{\tau_1}$ according to the probability $\dBg_{\beta}(\cX^{(\beta)}_0,\cdot)/\vert \dBg_{\beta}(\cX^{(\beta)}_0,\cX^{(\beta)}_0)\vert$ on $\{\cX^{(\beta)}_0-1, \cX^{(\beta)}_0+1\}$.
Next the same procedure starts again: sample an independent exponential time $\tau_2$ of parameter $\vert \dBg_{\beta}(\cX^{(\beta)}_{\tau_1},\cX^{(\beta)}_{\tau_1})\vert$
and take $\cX^{(\beta)}_t\df \cX^{(\beta)}_{\tau_1}$ for $t\in[\tau_1,\tau_1+\tau_2)$. Choose $\cX^{(\beta)}_{\tau_1+\tau_2}$ according to the probability $\dBg_{\beta}(\cX^{(\beta)}_{\tau_1},\cdot)/\vert \dBg_{\beta}(\cX^{(\beta)}_{\tau_1},\cX^{(\beta)}_{\tau_1})\vert$ on $\{\cX^{(\beta)}_{\tau_1}-1, \cX^{(\beta)}_{\tau_1}+1\}$ etc\ldots\par
\sm

In particular, we have
\bq
\Bs^{(\beta)}_{1+t}&=&\Bs^{(\beta)}_{1}\Bs^{(\beta)}_{t}\\
&=&\Lambda \dBs^{(\beta)}_t\adjL. \eq
As a consequence, for any given $x\geq 0$, to simulate $X^{(\beta)}_{1+t}$ under $\PP_x$, i.e.\ to sample according to $\Bs^{(\beta)}_{1+t}(x,\cdot)$ (it corresponds to following the path in blue in Figure \ref{fig3}, with $s=1$),
it is sufficient to sample $\cX_0^{(\beta)}$ according to  $\Lambda(x,\cdot)$, which is just the Poisson distribution on $\ZZ_+$ of parameter $x$,
to construct the evolution $(\cX^{(\beta)}_u)_{u\in[0,t]}$, as explained before Proposition \ref{semig} and finally to sample a point $Z$ according to $\adjL (\cX^{(\beta)}_t,\cdot)$, which is the gamma distribution of shape $\cX^{(\beta)}_t$ and of scale 1 (this procedure is colored in red in Figure \ref{fig3}). The distribution of $Z$ is exactly $\Bs^{(\beta)}_{1+t}(x,\cdot)$ and the complexity of this procedure is very simple.\par
\begin{rem}\label{Makarov}
Note that Makarov and Glew \cite{MR2747817} proposed a similar procedure for sampling with respect to $\Bs^{(\beta)}_{t}(x,\cdot)$, for any $x\in\RR_+$ and $t\in\RR_+$, but without using
the birth and death process $\cX^{(\beta)}$, see also Lemma \ref{LcLc} below.
\end{rem}
\par\me

\subsection{A limit theorem by intertwinning}\label{altbi}
One advantage of Lemmas \ref{LcLc1} and \ref{cinter} is to facilitate discrete approximations of the Bessel processes $X^{(\beta)}$ by birth-and-death processes.
Traditionally, such an approximation can be constructed in the following way.
Fix some $\epsilon>0$. To any function $f\in \esC_0(\RR_+)$, associate $\wi T_\epsilon f\in \esc_0(\ZZ_+)$ via
\bq
\fo n\in\ZZ_+,\qquad \wi T_\epsilon f(n)&\df& f(\epsilon n)\eq
and conversely, to any
$g\in \esc_0(\ZZ_+)$, associate $ T_\epsilon g\in \esC_0(\RR_+)$ via
\bq
\fo x\in\RR_+,\qquad  T_\epsilon g(x)&\df& g(\lfloor x/\epsilon \rfloor).\eq
When $X^{(\beta)}_0=\epsilon n$ with $n\in\ZZ_+$, consider
\bq
\tau_\epsilon&\df& \inf\{t\geq 0\st X^{(\beta)}_t\in\{\epsilon(n-1),\epsilon(n+1)\}\}\eq
and  the birth-and-death generator $\dBg_\epsilon^{(\beta)}$ defined by
\bq
\fo n\not= m\in\ZZ_+,\qquad
\dBg_\epsilon^{(\beta)}(n,m)&=&\lt\{
\begin{array}{ll}
\PP_{\epsilon n}[X_{\tau_\epsilon}^{(\beta)}=\epsilon m]/\EE_{\epsilon n}[\tau_\epsilon]&\hbox{ if $m\in\{n-1,n+1\}$}\\
0&\hbox{ otherwise.}\end{array}\rt.\eq
Let $(\exp(t\dBg_\epsilon^{(\beta)}))_{t\geq 0}$ be the semi-group generated by $\dBg_\epsilon^{(\beta)}$ in $ \esc_0(\ZZ_+)$.
It can be expected that for any time $t\geq 0$ and any function $f\in \esC_0(\RR_+)$, we have in the supremum norm of $\esC_0(\RR_+)$,
\bq
\lim_{\epsilon\ri 0_+}
 T_\epsilon \exp(t\dBg_\epsilon^{(\beta)}) \wi T_\epsilon f&=&
\Bs^{(\beta)}_t f.\eq
Nevertheless the rigorous proof of this approximation result is quite technical.
Up to replacing $\wi T_\epsilon$, $ T_\epsilon$ and $\dBg_\epsilon^{(\beta)}$ respectively by $\wi\Lambda_{1/\epsilon}$, $\Lambda_{1/\epsilon}$ and $\epsilon^{-1}\dBg^{(\beta)}$, Lemmas \ref{LcLc1}  and \ref{cinter} enable to simplify considerably such approximations.
Indeed, we have, in the supremum norm of $\esC_0(\RR_+)$, for any time $t\geq 0$ and any function $f\in \esC_0(\RR_+)$,
\bq
\lim_{\epsilon\ri 0_+}
\Lambda_{1/\epsilon} \dBs_{\epsilon^{-1}t}^{(\beta)} \wi\Lambda_{1/\epsilon} f&=&
\lim_{\epsilon\ri 0_+}
\Bs^{(\beta)}_t \Lambda_{1/\epsilon}\wi\Lambda_{1/\epsilon}f=\lim_{\epsilon\ri 0_+}
\Bs^{(\beta)}_{t+\epsilon}f=
\Bs^{(\beta)}_t f.\eq

\section{Classical and birth-and-death Laguerre processes}\label{cabadLp}

\subsection{Classical Laguerre semi-groups}\label{cLp}

Recall that, up to an isomorphism, the affine group acting on $\RR$ is the set
 $\RR\times (\RR\setminus\{0\})$ endowed with the operation $\ltimes$ defined by
\bq
\fo (u,v), (u',v')\in\RR\times (\RR\setminus\{0\}),\qquad (u,v)\ltimes (u',v')&\df& (u+v^{-1}u', vv').\eq
Consider $\cS\df \RR_+\times (0,\iy )$, which is stable by $\ltimes$ and so $(\cS, \ltimes)$ is a semi-group.
For any $\sigma>0$, define
\bq
\cS_\sigma&\df& \{(\sigma(e^{t}-1),e^{-t})\st t\in\RR_+\}\ \subset\ \cS.\eq
It is immediate to check that $(\cS_\sigma,\ltimes)$ is a one-dimensional sub-semi-group of the two dimensional semi-group $(\cS, \ltimes)$.
\par
Fix $\beta>0$, and  recall that for any $(u,v)\in\cS$,  $Q^{(\beta)}_u$ and $d_v$ belong to $\cB(\esL^2(\mu_{\beta}))$, the space of operators  bounded in $\esL^2(\mu_{\beta})$.
Denote
\bq
\fo (u,v)\in \cS,\qquad S^{(\beta)}_{u,v}&\df& Q^{(\beta)}_ud_v\ \in\ \cB(\esL^2(\mu_{\beta})).\eq
Endowing $\cB(\esL^2(\mu_{\beta}))$ with the usual composition operation, we see that the above mapping is a semi-group morphism,
as a consequence of the scaling intertwining relation
\bqn{scal}
\fo (u,v)\in\cS,\qquad Q^{(\beta)}_ud_v&=&d_vQ^{(\beta)}_{vu}.\eqn
It follows that  $(S^{(\beta)}_{(u,v)})_{(u,v)\in\cS}$ is a semi-group, indexed by a two-dimensional parameter.
Restricting the index set to $\cS_\sigma$, for some given $\sigma>0$, we write $(K^{(\beta,\sigma)}_t)_{t\in\RR_+}$ for the semi-group
given by
\bq
\fo t\geq 0,\qquad K^{(\beta,\sigma)}_t&\df& S_{(\sigma(e^{t}-1),e^{-t})}.\eq
Let us proceed by computing the generator of this continuous semi-group (still in $\esL^2(\mu_{\beta})$).
First remark that $(d_{e^{-t}})_{t\geq 0}$ is also a continuous semi-group  in $\esL^2(\mu_{\beta})$
and that its generator is given by $D$, the closure of the operator defined, $\fo f\in \esPe$ and  $\fo x\in\RR_+$, by
\bqn{def:D}
D f(x)&\df& -x\pa f(x)\eqn
(we hope the notation $D$ for the  generator of $(d_{e^{-t}})_{t\geq 0}$ is not confusing).
By differentiating
$K^{(\beta,\sigma)}_t$ at $t=0_+$ in $\esL^2(\mu_{\beta})$, we get that the generator $L_{\beta,\sigma}$ of $K^{(\beta,\sigma)}$ is the closure
of the operator $\sigma G_{\beta}+D$ acting on $\esPe$, where we recall that $G_{\beta}$ is the generator of the semi-group $(Q^{(\beta)}_t)_{t\in\RR_+}$. In particular,
$L_{\beta,\sigma}$ acts on functions $f\in \esPe$ via
\bqn{Lsigma}
\fo x\in\RR_+,\qquad
L_{\beta,\sigma}f(x)&=& \sigma x\pa^2f(x) +(\sigma\beta -x)\pa f(x)\eqn
and is the Laguerre differential operator.
This operator is a one-dimensional diffusion generator and to get a corresponding reversible measure it is sufficient
to compute the associated speed measure. Up to a positive multiplicative factor, its density with respect to the Lebesgue measure on $\RR_+$ is given by
\bq
\RR_+\ni x&\mapsto& \frac1{\sigma x}\exp\lt(\int_1^x \frac{\sigma\beta-y}{\sigma y}\, dy\rt)\\
&=&\frac1{\sigma x}\exp\lt(\beta \ln(x) -\frac{x-1}{\sigma }\rt)\\
\eq
(see for instance Chapter 15 of the book of Karlin and Taylor
\cite{MR611513}). It appears that this speed measure has a finite mass and that it can be normalized into the probability measure $\nu_{\sigma}$,
which is the gamma distribution of shape parameter $\beta$ and scale parameter $\sigma$, i.e.
\bq
\fo x\in(0,\iy ),\qquad \nu_{\sigma}(dx)&=& \frac{1}{\sigma^{\beta}\Gamma(\beta)}x^{\beta-1}\exp(-x/\sigma)\, dx.\eq
It follows (e.g.\ via Freidrichs theory) that the restriction of $L_{\beta,\sigma}$ to $\esPe$ can be extended into a self-adjoint operator on $\esL^2(\nu_{\sigma})$
and that the corresponding Markov semi-group, still denoted $K^{(\beta,\sigma)}$, is  continuous on $\esL^2(\nu_{\sigma})$.
\begin{rem}\label{Q0dot}
The  probability measure $\nu_{\sigma}$ coincides with
$Q^{(\beta)}_{\sigma}(0,\cdot)$, seen as the regular Markov kernel associated to the
$\esC_0(\RR_+)$-semi-group $(Q^{(\beta)}_t)_{t\in\RR_+}$ (note that this point of view
also leads to an interpretation of the semi-group $(K^{(\beta,\sigma)}_t)_{t\in\RR_+}$ in $\esC_0(\RR_+)$).
Indeed, for any $f\in\esC_0(\RR_+)$, we have for any $t\geq 0$,
\bq
Q^{(\beta)}_\sigma K_t^{(\beta,\sigma)}f(0)&=&Q^{(\beta)}_\sigma Q^{(\beta)}_{\sigma(e^{t}-1)}d_{e^{-t}} f(0)
=Q^{(\beta)}_{\sigma e^{t}}d_{e^{-t}}f(0)
=d_{e^{-t}}Q^{(\beta)}_\sigma f(0)\\
&=&Q^{(\beta)}_\sigma f(e^{-t}0)=Q^{(\beta)}_\sigma f (0).\eq
Since $\esC_0(\RR_+)$ is a core for $\esL^1(\nu_{\sigma})$, we get that $Q^{(\beta)}_{\sigma}(0,\cdot)$ is invariant for the Markov semi-group $K^{(\beta,\sigma)}$.
By irreducibility of the generator $L_{\beta,\sigma}$, there exists at most one  such invariant probability,
leading to $\nu_{\sigma}=Q^{(\beta)}_{\sigma}(0,\cdot)$.
\par
Another way to obtain this equality, is to see $Q^{(\beta)}_{\sigma}(0,\cdot)$ as the limiting distribution for large times of the semi-group
$K^{(\beta,\sigma)}$. Indeed, or any $f\in\esC_0(\RR_+)$ and any $x\in \RR_+$, we have
\bq
\lim_{t\ri\iy }
\delta_x K_t^{(\beta,\sigma)}f&=&
\lim_{t\ri\iy }K_t^{(\beta,\sigma)}f(x)=\lim_{t\ri\iy }Q^{(\beta)}_{\sigma(e^{t}-1)}d_{e^{-t}}f(x)
=\lim_{t\ri\iy }d_{e^{-t}}Q^{(\beta)}_{\sigma(1-e^{-t})}f(x)
\\ &=&\lim_{t\ri\iy }Q^{(\beta)}_{\sigma(1-e^{-t})}f(e^{-t}x)\\
&=&Q^{(\beta)}_\sigma f(0)\eq
(for the penultimate equality, we used that with respect to the operator supremum norm in $\esC_0(\RR_+)$, the difference $Q^{(\beta)}_{\sigma(1-e^{-t})}-Q^{(\beta)}_\sigma$ converges to zero for large $t\geq 0$).
By dominated convergence, it follows that for any probability measure $p$ on $\RR_+$, we have the weak convergence of
$pK_t$ toward $Q^{(\beta)}_{\sigma}(0,\cdot)$ for large $t\geq 0$. In particular, with $p=\nu_{\sigma}$, we recover that $\nu_{\sigma}=Q^{(\beta)}_{\sigma}(0,\cdot)$.
We deduce that \bq
\fo x\in(0,\iy ),\qquad
Q^{(\beta)}_{\sigma}(0,dx)&=& \frac{1}{\sigma^{\beta}\Gamma(\beta)}x^{\beta-1}\exp(-x/\sigma)\, dx\eq
namely $Q^{(\beta)}_{\sigma}(0,\cdot)$ admits the density $(0,\iy )\ni x\mapsto \exp(-x/\sigma)/\sigma^{\beta}$
with respect to $\mu_{\beta}$.
\end{rem}

\subsection{Discrete Laguerre semi-groups}\label{dLp}

The starting points of the considerations of the previous subsection were
the multiplicative semi-group property of $(d_v)_{v\in(0,\iy )}$  and
 the scaling intertwining \eqref{scal}, both in $\esL^2(\mu_{\beta})$ (where the parameter $\beta>0$ keeps being omitted).
 Here we will replace these  relations by their discrete analogues in order to deduce similar constructions.\par
\sm
As seen in Appendix \ref{otdco}, to get nice properties of the discrete dilatation operators, we cautiously restrict
our attention to the family $(\discret{D}_{v})_{v\in(0,1]}$ and introduce the
 sub-semi-group $\wi\cS$ of $\cS$ defined by
 \bq
 \wi\cS&\df& \{(u,v)\in\cS\st v\in (0,1]\}.\eq
 Indeed,  in addition of the multiplicative semi-group property of $(\discret{D}_v)_{v\in(0,1]}$, we have
 \bqn{dscal}
\fo (u,v)\in\wi\cS,\qquad \dBs^{(\beta)}_u\discret{D}_v&=&\discret{D}_v\dBs^{(\beta)}_{vu}.\eqn
Consider the semi-group $(\discret{S}^{(\beta)}_{u,v})_{(u,v)\in \wi\cS}$ of operators in $\cB(\esL^2(\fm))$ given by
\bq
\fo (u,v)\in \cS,\qquad \discret{S}^{(\beta)}_{u,v}&\df& \dBs^{(\beta)}_u\discret{D}_v\ \in\ \cB(\esLd^2(\fm_{\beta}))\eq
as well as, for $\sigma>0$, the sub-semi-group $\dLs^{(\beta,\sigma)}\df(\dLs^{(\beta,\sigma)}_t)_{t\geq 0}$ defined by
\bq
\fo t\geq 0,\qquad \dLs^{(\beta,\sigma)}_t&\df& \discret{S}^{(\beta)}_{(\sigma(e^{t}-1),e^{-t})}.\eq
\par
We compute the generator $\dLg_{\beta,\sigma}$ of $\dLs^{(\beta,\sigma)}$ in $\esLd^2(\fm_{\beta})$ exactly as in Subsection \ref{cLp}, taking into account Lemma \ref{appl2} of Appendix \ref{otdco}.
We get that $\dLg_{\beta,\sigma}$ acts on functions $f$ belonging to the core $\esFf$ via
\bqn{LG}
\fo n\in\ZZ_+,\qquad \dLg_{\beta,\sigma} f (n)&=&\sigma \dBg_{\beta} f(n)+\discret{D} f (n)\eqn
where the operator $\discret{D}$ is defined in Lemma \ref{appl2}.
Thus the generator $\dLg_{\beta,\sigma}$
can  be represented by the infinite tri-diagonal matrix $(\dLg_{\beta,\sigma}(m,n))_{m,n\in\ZZ_+}\df (\dLg_{\beta,\sigma} \un_{\{n\}}(m))_{m,n\in\ZZ_+}$, given explicitly by
\bqn{dLg}
\fo m,n\in \ZZ_+,\qquad
\dLg_{\beta,\sigma}(m,n)&=&\lt\{\begin{array}{ll}
(\sigma +1)m&\hbox{ if $n=m-1$}\\
-\sigma(2m+\beta)-m&\hbox{ if $n=m$}\\
\sigma(m+\beta)&\hbox{ if $n=m+1$}\\
0&\hbox{ otherwise.}
\end{array}
\rt.
\eqn
\par
A corresponding invariant measure $\wi\fn_{\sigma}$ is given by
\bq
\fo n\in\ZZ_+,\qquad
\wi\fn_{\sigma}(n)&\df&\prod_{m=0}^{n-1}\frac{\dLg_{\beta,\sigma}(m,m+1)}{\dLg_{\beta,\sigma}(m+1,m)}\\
&=&\lt(\frac{\sigma}{\sigma+1}\rt)^n\frac{(n+\beta-1)(n+\beta-1)\cdots \beta}{n!}\\
&=&\lt(\frac{\sigma}{\sigma+1}\rt)^n\fm_{\beta}(n).
\eq\par
Recognizing $\wi\fn_{\sigma}$ to be proportional to the negative binomial distribution of parameters $\beta$ and $\sigma/(1+\sigma)$, we get
\bq
 \sum_{n\in\ZZ_+}\wi\fn_{\sigma}(n)&=&\lt(1-\frac{\sigma}{1+\sigma}\rt)^{-\beta}\ =\ (1+\sigma)^{\beta}.\eq
From now on, we will rather consider the invariant probability $\fn_{\sigma}$ which is normalization of $\wi\fn_{\sigma}$, dividing it by  $(1+\sigma)^{\beta}$.
\par
It follows (e.g.\ via Freidrichs theory) that the restriction of $\dLg_{\beta,\sigma}$ to $\esFf$ can be extended into a self-adjoint operator on $\esLd^2(\fn_{\sigma})$
and that the corresponding Markov semi-group, still denoted $\dLs^{(\beta,\sigma)}$, is  continuous on $\esLd^2(\fn_{\sigma})$.
\par
\begin{rem} The arguments of Remark \ref{Q0dot} are still valid in the present setting and we deduce that for any $\sigma>0$,
\bq
\dBs^{(\beta)}_{\sigma}(0,\cdot)&=&\fn_{\sigma}\eq
i.e.\ $\dBs^{(\beta)}_{\sigma}(0,\cdot)$ admits the density $\ZZ_+\ni n\mapsto (1+\sigma)^{-\beta}(\sigma/(1+\sigma))^n$ with respect to $\fm_{\beta}$.
\end{rem}
\begin{rem}
In the previous two subsections, only the case $\beta,\sigma>0$ was considered. Indeed, for $\beta=0$ or $\sigma=0$, we have $\fn_{\beta,\sigma}=\delta_0$, whose $\esL^2$-theory is not really satisfactory! Furthermore, when $\beta=0$, $\esPe$ is not even included into $\esL^2(\mu_\beta)$.
Thus for $\beta=0$, it is more relevant to work with $\esC_0$ spaces, from an analytical point of view, or with Markov kernels, from a probabilist point of view.
\end{rem}

\subsection{Proof of Theorem \ref{eq:gateway-lag}: Gateway between Laguerre semi-groups}\label{iocadLp}

Recall that the Markov kernel $\Lambda$ from $\RR_+$ to $\ZZ_+$ was defined in \eqref{eq:Lambda-def} and that more generally for any $\sigma\geq 0$, we consider
the Markov kernel $\Lambda_\sigma=d_\sigma\Lambda$, meaning that for each $x\in\RR_+$, $\Lambda_\sigma(x,\cdot)$ is the Poisson distribution of parameter $\sigma x$.
The following result is an extension of Theorem \ref{thm:mainK}, in the spirit of Lemma \ref{cinter}. 
\begin{pro}\label{laginter}
Let $\beta,\sigma\geq 0$ be given, as well as $(u,v)\in \cS$.
We have on  $\esc_0(\ZZ_+)$
\bq
\Gate{S_{(u,v)}^{(\beta)}}{\ \Lambda_\sigma\ }{\discret{S}_{(\sigma u,v)}^{(\beta)}}.
 \eq
 In particular, for any $\beta,\varsigma,\sigma,t\geq 0$,
 we deduce that on  $\esc_0(\ZZ_+) \cup \esLd^2(\fn_{\varsigma \sigma})$
 \bq
\Gate{\Ls_{t}^{(\beta,\varsigma)}}{\ \Lambda_\sigma\ }{\dLs_{t}^{(\beta,\varsigma\sigma)}}.
 \eq
 This gateway relationship can be lifted to a unitary equivalence.
\end{pro}
\proof
From Lemmas \ref{lem:dLamD} and \ref{cinter}, we know, that for any $\beta,\sigma\geq 0$, on $\esc_0(\ZZ_+)$,
\bq
\fo u\geq 0, \qquad \Bs_u^{(\beta)}\Lambda_\sigma &=& \Lambda_\sigma \dBs^{(\beta)}_{\sigma u}\\
\fo v\geq 0,\qquad d_v\Lambda_\sigma &=&
d_vd_\sigma\Lambda
=d_\sigma d_v\Lambda
=D\sigma\Lambda\discret{D}_v
= \Lambda_\sigma\discret{D}_v.\eq
It follows that for any $(u,v)\in \cS$,
\bq
S_{(u,v)}^{(\beta)}\Lambda_\sigma&=&\Bs_u^{(\beta)}d_v\Lambda_\sigma=\Bs_u^{(\beta)}\Lambda_\sigma\discret{D}_v=\Lambda_\sigma\dBs_{\sigma u}^{(\beta)}\discret{D}_v
=\Lambda_\sigma\discret{S}^{(\beta)}_{(\sigma u,v)}.\eq
\par
The second announced intertwining relationship is obtained by taking $(u,v)=(\varsigma(e^{t}-1), e^{-t})$, for $\varsigma>0$ and $t\geq 0$ and its extension to $ \esLd^2(\fn_{\varsigma \sigma})$ follows by a density argument.  Finally, since by Lemma \ref{lem:lambda_bounded}, $\Lambda_{\sigma}: \esLd^2(\fn_{\varsigma \sigma}) \mapsto \esL^2(\nu_{\beta,\varsigma})$  is a quasi-affinity, we obtain from a  result of Douglas \cite{Douglas}  that there exists a unitary operator that intertwines $\Ls_{t}^{(\beta,\varsigma)}$ and $\dLs_{t}^{(\beta,\varsigma\sigma)}$.
\wwtbp

We also provide an alternative proof which is in the spirit of the one developed in Section \ref{cabadBp}. It stems on the following gateway relationship  between the generators of the discrete and continuous Laguerre semi-groups. Note that the lifting of this identity on $\esc_0(\ZZ_+) \cup \esLd^2(\fn_{\varsigma \sigma})$ between the corresponding semi-groups  could also be obtained from this result  by following a line of reasoning  similar  to the one developed in Section \ref{sec:proof_main}.
\begin{lemma}\label{lem:inter_xpart}
We have, on $\esP_{\fe_{-}}$,
\begin{equation}\label{eq:com_int}
  \Gate{ \discret{D}}{\na}{ D}
\end{equation}
where we have set $\discret{D}f (n)= n \delta_-f (n)$. Consequently, for any $\beta,\sigma>0$,
\begin{equation}\label{eq:com_int_L}
  \Gate{\discret{L}_{\beta,\sigma} }{\na}{\cont{L}_{\beta,\sigma}}.
\end{equation}

\end{lemma}
\begin{proof}
Since plainly $ D(\esP_{\fe_{-}})\subseteq \esP_{\fe_{-}}$, one has, on $\esP_{\fe_{-}}$, that
\begin{eqnarray*}
    -\na D f(n) &=& \partial^n(\fe x\partial f)(0)
     = n\partial^{n-1}\fe \partial f(0)
   =n \sum_{k=0}^{n-1} {n-1 \choose k}\partial^{k+1}f(0)
\end{eqnarray*}
and
\begin{eqnarray*}
   -\delta_-\na f(n) &=& (\partial^{n}\fe f)(0)-(\partial^{n-1}\fe f)(0)\\
   &=&
   \sum_{k=1}^{n} {n \choose k}\partial^{k}f(0)-\sum_{k=1}^{n-1} {n-1 \choose k}\partial^{k}f(0)\\
   &=& \partial^{n}f(0) + \sum_{k=1}^{n-1} {n-1 \choose k-1}\partial^{k}f(0)\\
   &=&    \sum_{k=0}^{n-1} {n-1 \choose k}\partial^{k+1}f(0)
\end{eqnarray*}
which by linearity completes the proof of the first identity. Next, invoking the  relation \eqref{Lsigma}  and Lemmas \ref{lem:lambda_range} and \ref{lem:inter_xpart}, we  get that, on $\esP_{\fe_{-}}$,
\begin{equation}\label{eq:com_int}
  \na \cont{L}_{\beta,\sigma} =  \sigma \na \cont{G}_{\beta}+ \na D = \sigma   \discret{G}_{\beta} \na + \discret{D}\na  = \discret{L}_{\beta,\sigma} \na
\end{equation}
where we also used \eqref{LG}, which completes the proof of the Lemma.
\end{proof}

\subsection{Dual gateway relationship and product of intertwining kernels}

In this part, we focus on the  $\esL^2$ interpretation of the above result, that is by viewing  the Markov kernels
$S_{(u,v)}^{(\beta)}$ and $K_t^{(\beta,\varsigma)}$ (respectively $\discret{S}_{(u,v)}^{(\beta)}$ and $\dLs_t^{(\beta,\varsigma\sigma)}$) as bounded operators on $\esL^2(\mu_\beta)$ and $\esL^2(\nu_{\beta,\varsigma})$ (resp.\ $\esLd^2(\fm_{\beta})$ and $\esLd^2(\fn_{\beta,\varsigma\sigma})$),
where we must restrict our attention to $\beta,\varsigma,\sigma >0$ and $(u,v)\in\wi \cS$.
Similarly, $\Lambda_\sigma$ should be seen as bounded operators from $\esLd^2(\fn_{\beta,\varsigma\sigma})$ to $\esL^2(\nu_{\beta,\varsigma})$. Note that due to the Markov property, all the previous operators have their
operator norms equal to $1$. The interest of this point of view is that it is immediate to consider the adjoint operators.
For any $\beta,\varsigma,\sigma>0$ and $ t\geq 0$, the operator $K_t^{(\beta,\varsigma)}$ (respectively\ $\dLs_t^{(\beta,\varsigma\sigma)}$) is self-adjoint in $\esL^2(\nu_{\beta,\varsigma})$ (resp.\ $\esLd^2(\fn_{\beta,\varsigma\sigma})$).
Denote $\wit \Lambda_{\beta,\varsigma,\sigma}\st \esL^2(\nu_{\beta,\varsigma})\ri \esLd^2(\fn_{\beta,\varsigma\sigma})$ the adjoint operator of $\Lambda_\sigma$.
It should not be confounded with $\Lambda_{\beta,\sigma}^{*}\st \esL^2(\mu_{\beta})\ri \esLd^2(\fm_{\beta})$, the adjoint operator of $\Lambda_\sigma$
when the latter is acting from $\esLd^2(\fm_{\beta})$ to $\esL^2(\mu_{\beta})$. These operators are nevertheless linked.
\begin{lem}\label{differentsadjoints}
We have for any $\beta,\sigma> 0$,
\bq
\wit \Lambda_{\beta,\varsigma,\sigma}&=&\lt(\varsigma^{-1}+\sigma\rt)^\beta \Lambda_{\beta,\varsigma^{-1}+\sigma}^{*}\ =\ \wi \Lambda_{\beta,\varsigma^{-1}+\sigma}\ =\ \wi \Lambda_{\beta,\sigma}d_{\varsigma\sigma/(1+\varsigma\sigma)}\eq
at least on $\esFf$.
\end{lem}
\proof
Let $H_{\beta,\varsigma}$ be the density of $\nu_{\beta,\varsigma}$ with respect to $\mu_{\beta}$ and $\discret{H}_{\beta,\varsigma\sigma}$ be the density of $\fn_{\beta,\varsigma\sigma}$ with respect to $\fm_{\beta}$.
The notations $H_{\beta,\varsigma}$ and $\discret{H}_{\beta,\varsigma\sigma}$ will also stand for the multiplication operators by these functions.
We have seen in Subsections \ref{cLp} and \ref{dLp} that these functions are exponential:
 \bq
H_{\beta,\varsigma}\st \RR_+\ni x&\mapsto& \varsigma^{-\beta}\exp(-x/\varsigma)\\
\discret{H}_{\beta,\varsigma\sigma}\st \ZZ_+\ni n&\mapsto&(1+\varsigma\sigma)^{-\beta} \lt(\frac{\varsigma\sigma}{1+\varsigma\sigma}\rt)^n.
\eq
Consider two test functions, say $f\in\esFf$ and $g\in\esPe/H_{\beta,\varsigma}$, we observe that
\bq
\nu_{\beta,\varsigma} g\Lambda_{\sigma}f&=&\mu_{\beta}H_{\beta,\varsigma}g\Lambda_{\sigma}f
=\fm_{\beta}\Lambda_{\beta,\sigma}^*H_{\beta,\varsigma}gf
=\fn_{\beta,\varsigma\sigma}\Lambda_{\beta,\sigma}^* H_{\beta,\varsigma}gf/\discret{H}_{\beta,\varsigma\sigma}
\eq
so that we get \bq\wit \Lambda_{\beta,\varsigma,\sigma}&=&(\discret{H}_{\beta,\varsigma\sigma})^{-1}\Lambda_{\beta,\sigma}^*H_{\beta,\varsigma}.\eq
Recalling that $\Lambda_{\beta,\sigma}^*$ and $\wi \Lambda_{\beta,\sigma}$ are described by the kernels
\begin{equation}\label{eq:def_Ladj}
\fo n\in\ZZ_+,\,\fo x\in(0,\iy ),\qquad\lt\{
\begin{array}{rcl}
\Lambda_{\beta,\sigma}^*(n,dx)&=& \sigma^n\frac{x^{n+\beta-1}}{\Gamma(n+\beta)}\exp(-\sigma x)\, dx\\
\wi \Lambda_{\beta,\sigma}(n,dx)&=& \sigma^\beta\Lambda_{\beta,\sigma}^*(n,dx),\end{array}\rt.
\end{equation}
we get for any $n\in\ZZ_+$ and $ x\in(0,\iy )$,
\bq
\wit \Lambda_{\beta,\varsigma,\sigma}(n,dx)&=& (1+\varsigma\sigma)^\beta\lt(\frac{\varsigma\sigma}{1+\varsigma\sigma}\rt)^{-n}\sigma^n\frac{x^{n+\beta-1}}{\Gamma(n+\beta)} \exp(-\sigma x)\varsigma^{-\beta}\exp(-x/\varsigma)dx\\
&=&\wi \Lambda_{\beta,\varsigma^{-1}+\sigma}(n,dx).
\eq
The last equality of the above lemma is a consequence of the general relation
\bq
\fo \beta,\sigma,\gamma>0,\qquad \wi \Lambda_{\beta,\sigma}d_\gamma&=&\wi \Lambda_{\beta,\sigma\gamma^{-1}}\eq
which amounts to the change of variable $(0,\iy )\ni x\mapsto \gamma^{-1}x$ in the gamma integrals defining the kernel $ \wi \Lambda_{\beta,\sigma}$ (or formally $\wi \Lambda_{\beta,\sigma}d_\gamma=\Lambda^*d_{1/\sigma}d_\gamma=\Lambda^*d_{\gamma/\sigma}=\wi \Lambda_{\beta,\sigma\gamma^{-1}}$).
\wwtbp

We deduce the following supplementary relations.
\begin{pro} \label{prop:dual-gat}
 For any $\beta,\varsigma,\sigma>0$ and  $t\geq 0$, we have, on  $\esL^2(\nu_{\beta,\varsigma})$, the gateway relation
\bq
\Gate{\dLs_{t}^{(\beta,\varsigma\sigma)}}{\ \wit \Lambda_{\beta,\varsigma,\sigma}\ }{\Ls_{t}^{(\beta,\varsigma)}}\eq
and
\bq
\Lambda_\sigma\wit \Lambda_{\beta,\varsigma,\sigma}&=&\Ls_{\ln(1+1/(\varsigma\sigma))}^{(\beta,\varsigma)}\\
\wit \Lambda_{\beta,\varsigma,\sigma}\Lambda_\sigma&=&\dLs_{\ln(1+1/(\varsigma\sigma))}^{(\beta,\varsigma\sigma)}.
\eq
\end{pro}
Remark that the time $\ln(1+1/(\varsigma\sigma))$  is the same in the  right-hand side of the two last identities. By comparing  Lemmas
\ref{LcLc1} and \ref{LcLc}, we observe that  differs from the Bessel case.
\proof
Passing to the adjoint operators in the second  intertwining relation of Proposition \ref{laginter}, we get, on $\esL^2(\nu_{\beta,\varsigma})$,
\bq
\Gate{\dLs_{t}^{(\beta,\varsigma\sigma)*}}{\ \wit \Lambda_{\beta,\varsigma,\sigma}\ }{\Ls_{t}^{(\beta,\varsigma)*}}\eq
which is the first announced result, by self-adjointness of $\dLs_{t}^{(\beta,\varsigma\sigma)}$ and $\Ls_{t}^{(\beta,\varsigma)}$.
Taking into account Lemmas  \ref{differentsadjoints} and \ref{LcLc1}, we obtain
\bq
\Lambda_\sigma\wit \Lambda_{\beta,\varsigma,\sigma}&=&\Lambda_\sigma\wi \Lambda_{\beta,\sigma}d_{\varsigma\sigma/(1+\varsigma\sigma)}\\
&=&Q_{1/\sigma}^{(\beta)}d_{\varsigma\sigma/(1+\varsigma\sigma)}.\eq
Defining $t\df\ln(1+1/(\varsigma\sigma))$
so that
$ (\varsigma(e^t-1),e^{-t})=(1/\sigma,\varsigma\sigma/(1+\varsigma\sigma))$, it follows that
\bq
Q_{1/\sigma}^{(\beta)}d_{\varsigma\sigma/(1+\varsigma\sigma)}&=&K^{(\beta,\varsigma)}_{\ln(1+1/(\varsigma\sigma))}.\eq
Similarly, rather taking into account Lemma \ref{LcLc}, we compute that
\bq
\wit \Lambda_{\beta,\varsigma,\sigma}\Lambda_\sigma&=&\wi \Lambda_{\beta,\sigma}d_{\varsigma\sigma/(1+\varsigma\sigma)}\Lambda_\sigma\\
&=&\wi \Lambda_{\beta,\sigma}\Lambda_\sigma\discret{D}_{\varsigma\sigma/(1+\varsigma\sigma)}\\
&=&\dBs^{(\beta)}_1\discret{D}_{\varsigma\sigma/(1+\varsigma\sigma)}\\
&=&\dLs^{(\beta,\varsigma\sigma)}_{\ln(1+1/(\varsigma\sigma))}.
\eq
\wwtbp

The above relations are summarized in  the  Figure \ref{fig4}, valid for any $\beta,\varsigma,\sigma>0$ and $t\geq 0$, which is the Laguerre analogue of Figure \ref{fig3} in the Bessel setting.
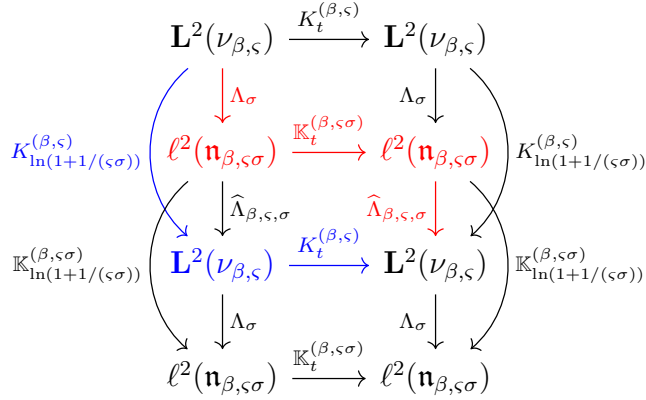
\begin{figure}[!h]\centering
\begin{tikzcd}[scale=2]
\textcolor{black}{\esL^2(\nu_{\beta,\varsigma}) }\arrow[r,  "\Ls^{(\beta,\varsigma)}_t"] \arrow[d, color=red,"\Lambda_{\sigma}"] \arrow[dd, bend right=50 , color=blue, "\Ls_{\ln(1+1/(\varsigma\sigma))}^{(\beta,\varsigma)}"']
& \esL^2(\nu_{\beta,\varsigma}) \arrow[d, "\Lambda_{\sigma}"' ]  \arrow[dd, bend left=50 , "\Ls_{\ln(1+1/(\varsigma\sigma))}^{(\beta,\varsigma)}"]\\
\textcolor{red}{\esLd^2(\fn_{\beta,\varsigma\sigma})} \arrow[r,color=red, "\dLs^{(\beta,\varsigma\sigma)}_{t}" ]\arrow[d, "\wit\Lambda_{\beta,\varsigma, \sigma}"]\arrow[dd, bend right=50 , "\dLs_{\ln(1+1/(\varsigma\sigma))}^{(\beta,\varsigma\sigma)}"']
& \textcolor{red}{\esLd^2(\fn_{\beta,\varsigma\sigma})} \arrow[d, color=red,"\wit\Lambda_{\beta,\varsigma,\sigma}"' ]\arrow[dd, bend left=50 , "\dLs_{\ln(1+1/(\varsigma\sigma))}^{(\beta,\varsigma\sigma)}"]\\
\textcolor{blue}{\esL^2(\nu_{\beta,\varsigma})} \arrow[r,color=blue,  "\Ls^{(\beta,\varsigma)}_t"] \arrow[d, "\Lambda_{\sigma}" ]
& \textcolor{black}{\esL^2(\nu_{\beta,\varsigma})} \arrow[d, "\Lambda_{\sigma}"' ]\\
\esLd^2(\fn_{\beta,\varsigma\sigma}) \arrow[r, "\dLs^{(\beta,\varsigma\sigma)}_{t}"]
& \esLd^2(\fn_{\beta,\varsigma\sigma})
\end{tikzcd}
\caption{Laguerre intertwining relations}\label{fig4}
\end{figure}\par
\sm
The  intertwining relations displayed in Figure \ref{fig4} admit three probabilistic consequences.
To state them, for any $\beta,\sigma>0$,
let $X^{(\beta,\sigma)}\df(X_t^{(\beta,\sigma)})_{t\geq 0}$ (respectively $\fX^{(\beta,\sigma)}\df(\fX_t^{(\beta,\sigma)})_{t\geq 0}$) be a Markov process associated to the generator $\Lg_{\beta,\sigma}$ (resp.\ $\dLg_{\beta,\sigma}$).\par\sm
$\bullet$ \textbf{Simulation}: for any $\beta,\sigma,\varsigma>0,\,t\geq 0$ and $x\in \RR_+$ (resp.\ $n\in\ZZ_+$), the random variable $X^{(\beta,\varsigma)}_{\ln(1+1/(\sigma\varsigma))+t}$ (resp.\ $\fX^{(\beta,\varsigma\sigma)}_{\ln(1+1/(\sigma\varsigma))+t}$) can be simulated in the following way,
when $X^{(\beta,\varsigma)}$ (resp.\ $\fX^{(\beta,\varsigma\sigma)}$) is starting from $x$ (resp.\ $n$).
First sample $n$ (resp.\ $x$) under the probability $\Lambda_\sigma(x,\cdot)$ (resp.\ $\wit\Lambda_{\beta,\varsigma,\sigma}(n,\cdot)$), next simulate $\fX^{(\beta,\varsigma\sigma)}_t$
(resp.\ $X^{(\beta,\varsigma)}_t$) starting from $n$ (resp.\ $x$),
and finally get  $X^{(\beta,\varsigma)}_{\ln(1+1/(\sigma\varsigma))+t}$ (resp.\ $\fX^{(\beta,\varsigma\sigma)}_{\ln(1+1/(\sigma\varsigma))+t}$)
by sampling with respect to $\wit\Lambda_{\beta,\varsigma,\sigma}(\fX^{(\beta,\varsigma\sigma)}_t,\cdot)$ (resp.\ $\Lambda_\sigma(X^{(\beta,\varsigma)}_t,\cdot)$).
This assertion is an immediate generalization of the observation made before Remark \ref{Makarov} and corresponds to the commutation of the paths in blue and red in Figure \ref{fig4}.
\par\sm
$\bullet$ \textbf{Approximation}:
for large $\sigma>0$,
 the birth and death process $(\fX_t^{(\beta,\sigma\varsigma)})_{t\geq 0}$ provides a convenient approximation of $(X_t^{(\beta,\varsigma)})_{t\geq 0}$, up to natural scalings,
 as in Subsection~\ref{altbi}.
 Indeed, we have for any bounded and continuous  function $f\st\RR_+\ri\RR$, $x\in\RR_+$ and $t\geq 0$,
\bq
 \Lambda_\sigma\dLs_t^{(\beta,\varsigma\sigma)}\wit\Lambda_{\beta,\varsigma,\sigma} f(x)&=&
 \Ls_t^{(\beta,\varsigma)}\Lambda_\sigma\wit\Lambda_{\beta,\varsigma,\sigma} f(x)\\
 &=& \Ls_t^{(\beta,\varsigma)}\Ls_{\ln(1+1/(\varsigma\sigma))}^{(\beta,\varsigma)}f(x)\\
 &=&\Ls_{t+\ln(1+1/(\varsigma\sigma))}^{(\beta,\varsigma)}f(x)\\
 &\xrightarrow{\sigma\ri +\iy}&\Ls_t^{(\beta,\varsigma)}f(x)\eq
 where the last convergence is a consequence of the continuity of the trajectories of the diffusions associated to the semi-group $\Ls^{(\beta,\varsigma)}$.
The convergence
\bq
\lim_{\sigma\ri +\iy} \Lambda_\sigma\dLs_t^{(\beta,\varsigma\sigma)}\wit\Lambda_{\beta,\varsigma,\sigma}f&=&\Ls_t^{(\beta,\varsigma)} f\eq
can also be understood in $\esC_0(\RR_+)$ or in $\esL^2(\nu_{\beta,\varsigma})$ for $f$ in these spaces, by continuity of the corresponding semi-group $\Ls^{(\beta,\varsigma)}$.
The advantage of this
approach to the convergence of discrete approximation (say for a finite number of time-marginal distributions) is that it is very simple in comparison with other methods,
see for instance the original proof of Feller \cite{Feller-Gen}.\par\sm
$\bullet$ \textbf{Speed of convergence}: here we just sketch this application, since it will be investigated in a general Markovian framework in \cite{intertwining}, to which we refer for the proof of Proposition \ref{Bakry} below as a particular case.
\par
Recall that the
entropy of two probability measures $\pi'$ and $\pi$ defined on the same state space is given by
\bq
\Ent(\pi'\vert\pi)&\df&\lt\{
\begin{array}{ll}\di \int \ln\lt(\frac{d\pi'}{d\pi}\rt)\, d\pi'&\hbox{ if $\pi'\ll\pi$}\\
+\iy&\hbox{ otherwise}\end{array}\rt.
\eq
where $d\pi'/d\pi$ stands for the Radon-Nikodym density of $\pi'$ with respect to $\pi$.
\par
For any $\beta,\sigma,\varsigma>0$,
the speed of convergence of $X^{(\beta,\varsigma)}$ (resp.\  $\fX^{(\beta,\varsigma\sigma)}$) toward the equilibrium $\nu_{\beta,\varsigma}$ (resp.\ $\fn_{\beta,\varsigma\sigma}$) in the entropy sense can be deduced from
the corresponding speed for $\fX^{(\beta,\varsigma\sigma)}$ (resp.\ $X^{(\beta,\varsigma)}$).
For instance, taking into account that Bakry \cite{MR1417973} has shown that the logarithmic Sobolev constant of the Laguerre generator $\Lg^{(\beta,\sigma)}$ is equal to 2 for
all $\sigma >0$ and $\beta\geq 1/2$, it is possible to deduce from the three last lines of Figure~\ref{fig4} the following result.
\begin{cor}\label{Bakry}
For any initial probability $m_0$  on  $\ZZ_+$ and for any $\beta\geq 1/2$, $\sigma>0$ and $t\geq 0$, we have
\bq
\Ent( m_0\dLs_t^{(\beta,\sigma)}\vert\fn_{\beta,\sigma})&\leq & \exp(-2[t-\ln(1+1/\sigma)]_+)\Ent( m_0\vert\fn_{\beta,\sigma}).
 \eq
\end{cor}
\par
So
up to waiting a warm-up time $\ln(1+1/\sigma))$,
 we get  an exponential rate of convergence equal to 2,  which the best possible one, since it corresponds to twice the spectral gap of $\dLg_{\beta,\sigma}$
 (see e.g.\ the book of Ané et al.\ \cite{MR2002g:46132}).
 The accuracy of Corollary \ref{Bakry} seems out-of-reach by directly working with the birth-death semi-group $(\dLs_t^{(\beta,\sigma)})_{t\geq 0}$.
Furthermore, the fact that the warm-up time $\ln(1+1/\sigma)$ vanishes as $\sigma$ goes to $+\iy$ is related to the approximation, mentioned above, of $X^{(\beta,1)}$ by $\fX^{(\beta,\sigma)}$ for $\sigma$ large.
\subsection{Spectral decomposition of the discrete Laguerre semi-group} \label{sec:spect_lag}

In this part, we show how the gateway relationship of Theorem  \ref{thm:mainK} can be used to recover the spectral decomposition in a weighted Hilbert space of
the discrete Laguerre semi-group from the semi-group of the continuous one. For sake of simplicity we assume that $\sigma^2=1$ and we denote simply $\Ls^{(\beta)}=\Ls^{(\beta,1)}$. Next, it is well-known, see e.g.~\cite{Bakry_Book}, that, for any $t>0$,  $\Ls_t^{(\beta)}$ is a Hilbert-Schmidt operator in $\esL^2(\nu_{\beta})$ that  admits, for any $f \in  \esL^2(\nu_{\beta})$ the diagonalization
\begin{equation}\label{eq:seK}
\Ls^{(\beta)}_t f(x) = \sum_{k=0}^{\infty}e^{-k t} \mathfrak{c}_{k}(\beta) \langle f, \mathcal{L}^{(\beta)}_k \rangle_{\nu_{\beta}}\:  \mathcal{L}^{(\beta)}_k(x)
\end{equation}
where  the sequence of Laguerre polynomials $(\sqrt{\mathfrak{c}_{k}(\beta)}\mathcal{L}^{(\beta)}_k)_{k\geq 0}$ forms an orthonormal basis in $\esL^2(\nu_{\beta})$ and we recall that
\begin{equation*} 
\mathcal{L}^{(\beta)}_k(x)=\sum_{r=0}^k (-1)^r { k +\beta \choose k-r}  \frac{x^r }{r!}
\end{equation*}
and $\mathfrak{c}_{k}(\beta)=\frac{\Gamma(k+1) \Gamma(\beta+1)}{\Gamma(k+\beta+1)}$.
Moreover,  the spectral theory of reversible Markov semigoups yields, for any $t\geq 0$ and $f\in \Lnu$, the spectral gap estimate
 \begin{equation}\label{eq:inv_Lag}
  \textrm{var}_{\nu_{\beta}}\left(\cont{K}^{(\beta)}_tf\right)\leq e^{-t} \: \textrm{var}_{\nu_{\beta}}\left(f\right)
 \end{equation}
where for a measure $\nu$, $\textrm{var}_{\nu}\left(f\right)=||f-\nu f||_{\L^2(\nu)}$.
We have the following analogue spectral theoretical result of the discrete semi-group.

\begin{proposition}
For all $ g \in  \esLd^2(\fn_{\beta}) $ and $t>0$, we have in $ \esLd^2(\fn_{\beta}) $,
 \begin{equation*}
   \Ks^{(\beta)}_t  g  = \sum_{k=0}^{\infty}e^{-kt}2^{k}\mathfrak{c}_{k}(\beta) \langle  g , \eigL^{(\beta)} \rangle_\meKp \eigL^{(\beta)}
 \end{equation*}
where $
\eigL^{(\beta)}(n)=\sum_{r=0}^{k} (-1)^r  { k \choose r} \frac{\Gamma(n+\beta+r)}{r! \Gamma(n+\beta)}. $
Finally, for any $ g \in  \esLd^2(\fn_{\beta}) $ and $t>0$, we have
\begin{equation*}
  \textrm{var}_{\fn_{\beta}}\left(\discret{K}^{(\beta)}_tg \right)\leq e^{-t} \: \textrm{var}_{\fn_{\beta}}\left(g\right).
 \end{equation*}
\end{proposition}
\begin{proof}
First, writing simply $\wit \Lambda_{\beta}=\wit \Lambda_{\beta,1,1}$, we note, from  \eqref{eq:def_Ladj}, that, for any $k,n\in \NN$,
\bqn{eq:Ll} \nonumber \wit \Lambda_{\beta}\mathcal{L}^{(\beta)}_k(n)&=& \int_0^{\infty}\mathcal{L}^{(\beta)}_k(x)\frac{x^{n+\beta-1}}{\Gamma(n+\beta)}\exp(- x)\, dx= \sum_{r=0}^k (-1)^r { k +\beta \choose k-r}  \frac{\Gamma(n+\beta+r)}{r! \Gamma(n+\beta)}\\ &=& \eigL^{(\beta)}(n). \eqn
Next, using  the gateway relationship stated in Proposition \ref{prop:dual-gat} with  $\varsigma=\sigma=1$, that is $\Gate{\dLs_{t}^{(\beta)}}{\ \wit \Lambda_{\beta}\ }{\Ls_{t}^{(\beta)}}$, we get, that for all $k \in \mathbb{N}$,
\begin{equation*}\label{}
   \Ks^{(\beta)}_t \eigL^{(\beta)} = \Ks^{(\beta)}_t \wit \Lambda_{\beta} \mathcal{L}^{(\beta)}_k=  \wit \Lambda_{\beta} \Ls^{(\beta)}_t\mathcal{L}^{(\beta)}_k = e^{- k t } \wit \Lambda_{\beta} \mathcal{L}^{(\beta)}_k= e^{- k t } \eigL^{(\beta)}
\end{equation*}
where we used that the Laguerre polynomials $ \mathcal{L}^{(\beta)}_k$ are eigenfunctions of $\Ls^{(\beta)}_t$ associated to the eigenvalues $e^{-kt}$. Next, since from Proposition \ref{prop:dual-gat}, we have that $\Ks^{(\beta)}_{\ln 2}=\wit \Lambda_{\beta} \Lambda$, the semi-group property of $\Ks^{(\beta)}$ entails that for any $ g \in  \esLd^2(\fn_{\beta}) $ and $t>\ln 2$, \[\Ks^{(\beta)}_t g  =  \Ks^{(\beta)}_{t-\ln 2}\wit \Lambda_{\beta} \Lambda f.\] Thus,  by means again of  the gateway relationship stated in Proposition \ref{prop:dual-gat}, the spectral expansion of $\Ls^{(\beta)}_t $ in \eqref{eq:seK} and the identity \eqref{eq:Ll} combined with the fact that $\wit \Lambda_{\beta}$ is bounded in $\esL^2(\nu_{\beta})$, as the adjoint of a bounded linear operator, we get that, for any $ g \in  \esLd^2(\fn_{\beta}) $ and $t>\ln 2$,
\begin{eqnarray*}
 \Ks^{(\beta)}_t g  &=&   \wit \Lambda_{\beta} \Ls^{(\beta)}_{t-\ln 2}  \Lambda g \\&=& \sum_{k=0}^{\infty}e^{-k t} 2^{k}\mathfrak{c}_{k}(\beta) \langle  \Lambda g, \mathcal{L}^{(\beta)}_k \rangle_{\nu_{\beta}}\:  \eigL^{(\beta)}\\&=& \sum_{k=0}^{\infty}e^{-k t} 2^{k}\mathfrak{c}_{k}(\beta) \langle g, \eigL^{(\beta)} \rangle_{\fn_{\beta}}\:  \eigL^{(\beta)}
\end{eqnarray*}
where for the last line we used another time the identity \eqref{eq:Ll}. To get the eigenvalues expansion for all $t>0$,  we first note that the Stirling formula yields that for $n$ large, $\fn_\beta(n)= 2^{-n-\beta}\frac{\Gamma(\beta+n)}{n!\Gamma(\beta)}\sim C n^{\beta} e^{-n \ln 2}, C>0,$ and thus we deduce from Lemma \ref{lem:lambda_bounded} that $\Lambda : \ell^2({\fn_{\beta}}) \mapsto \esL^2(\nu_{\beta})$  is a quasi-affinity.  Hence, according to Douglas \cite{Douglas}, the gateway relationship \eqref{eq:gateway-lag} between the two self-adjoint Laguerre operators can be lifted to a unitary equivalence between these semi-groups, that is there exists an unitary operator $U: \esL^2(\nu_{\beta})\mapsto  \ell^2(\fn_{\beta})$ such that \begin{equation} \label{eq:U}
\Gate{\dLs_{t}^{(\beta)}}{U}{\Ls_{t}^{(\beta)}}.
\end{equation} This entails that for all $t>0$, $\Ks^{(\beta)}_t$ and $\Ls^{(\beta)}_t$ are isospectral and the former is also a  Hilbert-Schmidt operator in  $\esLd^2(\fn_{\beta})$. The semi-group property of $\Ks^{(\beta)}$ provides the spectral expansion for all $t>0$. To conclude the proof, we note, from \eqref{eq:U}, that for all $ g \in  \esLd^2(\fn_{\beta}) $, $\fn_{\beta} g=\nu_{\beta}U g$
 and hence for all $t>0$
\begin{eqnarray*}
 ||\Ks^{(\beta)}_tg-\fn_{\beta} g||_{\ell^2(\fn_{\beta})} &=& ||U^{-1}\Ls^{(\beta)}_tUg-\nu_{\beta}U g||_{\ell^2(\fn_{\beta})} =  ||\Ls^{(\beta)}_tUg-\nu_{\beta}U g||_{\esL^2(\nu_{\beta})} \\
  &\leq & e^{-t} ||Ug-\fn_{\beta} g||_{\esL^2(\nu_{\beta})} =e^{-t} ||g-\fn_{\beta} g||_{\ell^2(\fn_{\beta})}
\end{eqnarray*}
which completes the proof.
\end{proof}


 \bibliographystyle{plain}

\appendix
\section{On the discrete contraction operators in $\esL^2$}\label{otdco}

 In analogy with the usual family $(d_\sigma)_{\sigma\in[0,1]}$,
the natural properties of the discrete contraction operators $(\discret{D}_\sigma)_{\sigma\in[0,1]}$ are presented here
in $\esLd^2(\fm_{\beta})$. In this appendix
the parameter $\beta>0$ is fixed and is consequently dropped from the notations, except when it is explicitly required by some expressions.
Follow some preliminary informations
 about these discrete contraction operators.
 \begin{lem}\label{dualdD}
 For any $\sigma\in [0,1]$, the operator $\discret{D}_{\sigma}$ is continuous on $\esL^2(\fm)$ and its operator norm is bounded above by ${\sigma}^{-\beta}$.
 Let $\discret{D}_{\sigma}^*$ be the dual operator of $\discret{D}_{\sigma}$ in $\LL^2(\fm)$,
 its kernel is given by
 \bq
 \fo m,n\in\ZZ_+,\qquad
 \discret{D}_{\sigma}^*(m,n)&=&{\sigma}^{-\beta}\binom{m+\beta-1}{m}\mathrm{NB}_{m+\beta,1-{\sigma}}(n-m)\eq
where $\mathrm{NB}_{m+\beta,1-{\sigma}}$ is the negative binomial distribution of parameters $m+\beta>0$ and $1-{\sigma}\in[0,1]$ (in particular  $\discret{D}_{\sigma}^*(m,n)=0$ when $n<m$)
and where by convention $\binom{m+\beta-1}{m}\df \Gamma(m+\beta)/(\Gamma(m+1)\Gamma(\beta))$, even when $\beta>0$ is not an integer number.\par
Furthermore, we have the following multiplicative semi-group property
\bq
\fo \sigma,{\sigma}'\in[0,1],\qquad \discret{D}_{\sigma}\discret{D}_{{\sigma}'}&=&\discret{D}_{{\sigma}{\sigma}'}.\eq
 \end{lem}
 \proof
 For the first assertion, fix $f\in\esL^2(\fm)$.
Taking into account that $\discret{D}_{\sigma}$ is a Markov kernel, we have
 \bq
 \fm(\discret{D}_{\sigma}f)^2&\leq & \fm\discret{D}_{\sigma}f^2
 =(\mu \Lambda)\discret{D}_{\sigma}f^2
 =\mu \Lambda \discret{D}_{\sigma}f^2
 =\mu d_{\sigma}\Lambda f^2
 \leq  {\sigma}^{-\beta}\mu \Lambda f^2
 ={\sigma}^{-\beta} \fm f^2 \eq
 where we used Lemma \ref{Linj} and \ref{lem:dLamD} and the fact that ${\sigma}^{\beta}d_{\sigma}$ is an isometry of $\esL^1(\mu)=\{f :\RR_+\mapsto \R \: \textrm{ measurable with } \int_{0}^{\infty}|f(x)| \mu(dx)<\infty\}$.
 \par
 Concerning the second assertion, fix two functions $f,g\in\esL^2(\fm)$.
 By definition of the dual operator $\discret{D}_{\sigma}^*$, we have
 \bq
 \fm g \discret{D}_{\sigma}^* f&=&\fm f\discret{D}_{\sigma} g\\
 &=&\sum_{n\in\ZZ_+}\fm(n)f(n)\sum_{m=0}^n \binom{n}{m} {\sigma}^m(1-{\sigma})^{n-m} g(m)\\
 &=&\sum_{m\in\ZZ_+}g(m)\sum_{n\geq m}\fm(n)\binom{n}{m} {\sigma}^m(1-{\sigma})^{n-m}f(n).\eq
 Since this holds for any $f,g\in\esL^2(\fm)$, we deduce that the kernel corresponding to $\discret{D}_{\sigma}^*$ is given
 by
 \bq
 \fo m,n\in\ZZ_+,\qquad
 \discret{D}_{\sigma}^*(m,n)&=&\fm(n)\binom{n}{m} {\sigma}^m(1-{\sigma})^{n-m}\\
 &=&\frac{\Gamma(n+\beta)}{n!\Gamma(\beta)}\frac{n!}{(n-m)! m!}{\sigma}^m(1-{\sigma})^{n-m}\\
 &=&\frac{\Gamma(m+\beta){\sigma}^m}{m!\Gamma(\beta){\sigma}^{m+\beta}}\frac{\Gamma(m+\beta+n-m)}{(n-m)!\Gamma(m+\beta)}(1-{\sigma})^{n-m}{\sigma}^{m+\beta}\\
 &=&{\sigma}^{-\beta}\binom{m+\beta-1}{m}\mathrm{NB}_{m+\beta,1-{\sigma}}(n-m)
 \eq
 recalling that for any parameters $r>0$ and $p\in(0,1)$, the negative binomial distribution $\mathrm{NB}_{r,p}$ is defined
 by
 \bq
 \fo n\in\ZZ_+,\qquad \mathrm{NB}_{r,p}(n)&\df& \binom{r+n-1}{n}p^n(1-p)^{r}.\eq
 \par
 Concerning the third assertion, let be given ${\sigma},{\sigma}'\in[0,1]$ and consider the product $\discret{D}_{\sigma}\discret{D}_{{\sigma}'}$, which has a meaning in $\cB(\esL^2(\fm))$
 according to the first point of this proof.
 Taking into account Lemma \ref{lem:dLamD}, we have
 \bq
 \Lambda \discret{D}_{\sigma}\discret{D}_{{\sigma}'}&=&d_{\sigma}\Lambda \discret{D}_{{\sigma}'}\\
 &=&d_{\sigma}d_{{\sigma}'}\Lambda\\
 &=&d_{{\sigma}{\sigma}'}\Lambda\\
 &=&\Lambda\discret{D}_{{\sigma}{\sigma}'}\eq
 and the injectivity property proved in Lemma \ref{Linj} enables us to conclude that $\discret{D}_{\sigma}\discret{D}_{{\sigma}'}=\discret{D}_{{\sigma}{\sigma}'}$.
 \wwtbp
 \par
 To get an additive semi-group, we rather consider  $(\discret{D}_{e^{-t}})_{t\geq 0}$.
 The next result computes its generator in $\esL^2(\fm)$.
\begin{lem}\label{appl2}
The  semi-group $(\discret{D}_{e^{-t}})_{t\geq 0}$ is continuous in $\esL^2(\fm)$ and its generator $\discret{D}$
acts on the core $\esFf$ via, for $n\in \ZZ_+$,
\bqn{dD}
\discret{D}f(n)&=& n\pa_-f(n)\\
&\df& n (f(n-1)-f(n))\eqn
(this term being 0 when $n=0$).
\end{lem}
\proof
 Due to the  semi-group property,
  it is sufficient
to check the continuity at $t=0$. Namely,
 we want to prove that for any $f\in\esL^2(\fm)$,
\bqn{contd}
\lim_{t\ri 0_+} \sum_{n\in\ZZ_+} \lt((\discret{D}_{e^{-t}}-\discret{D}_{\exp(-0)})f(n)\rt)^2\,\fm(n)&=&0.\eqn
Due to the continuity of the operators $\discret{D}_{e^{-t}}$ for $t\geq 0$ (and to the bounds on their norms given in Lemma~\ref{dualdD}),
it is enough to prove this convergence for $f$ in the core $\esFf$.
For given $n\in\ZZ_+$, we have
\bq
\lefteqn{\lt((\discret{D}_{e^{-t}}-\discret{D}_{\exp(-0)}) f(n)\rt)^2}\\&=&\lt((\discret{D}_{e^{-t}}-\discret{D}_{1})f(n)\rt)^2\\
&=&\lt(( e^{-nt}-1)f(n)+\sum_{m=0}^{n-1}\binom{n}{m} e^{-mt}(1-e^{-t})^{n-m}f(m)\rt)^2\\
&\leq & 2(( e^{-nt}-1)f(n))^2+2\lt(\sum_{m=0}^{n-1}\binom{n}{m} e^{-mt}(1-e^{-t})^{n-m}f(m)\rt)^2.\eq
Since $f(n)$ vanishes except for a finite number of $n\in\ZZ_+$, it appears that
\bq
\lim_{t\ri 0_+} \sum_{n\in\ZZ_+} ( e^{-nt}-1)^2f^2(n) \,\fm(n)&=&0\eq
so to get \eqref{contd}, we must prove that
\bqn{contd2}
\lim_{t\ri 0_+} \sum_{n\in\ZZ_+} \lt(\sum_{m=0}^{n-1}\binom{n}{m} e^{-mt}(1-e^{-t})^{n-m}f(m)\rt)^2\,\fm(n)&=&0.\eqn
Applying Cauchy-Schwarz inequality, we have for any $n\in\ZZ_+$,
\bq
 \lt(\sum_{m=0}^{n-1}\binom{n}{m} e^{-mt}(1-e^{-t})^{n-m}f(m)\rt)^2&\leq & g_t(n) \lt(\sum_{m=0}^{n}\binom{n}{m} e^{-mt}(1-e^{-t})^{n-m}f^2(m)\rt)
 \\&=& g_t(n)\discret{D}_{e^{-t}} f^2(n)
 \eq
where
\bq
\fo n\in\ZZ_+,\qquad g_t(n)&\df& \sum_{m=0}^{n-1}\binom{n}{m} e^{-mt}(1-e^{-t})^{n-m}\\
&=&1-\exp(-tn).\eq
We deduce
\bq
\sum_{n\in\ZZ_+} \lt(\sum_{m=0}^{n-1}\binom{n}{m} e^{-mt}(1-e^{-t})^{n-m}f(m)\rt)^2\,\fm(n) \leq \fm g_t\discret{D}_{e^{-t}} f^2 =\fm \discret{D}^*_{e^{-t}}g_t f^2.\eq
This expression leads us to compute $\discret{D}^*_{e^{-t}}g_t $ with the help of Lemma \ref{dualdD}. We have for any $n\in\ZZ_+$,
\bq
\discret{D}^*_{e^{-t}} g_t (n)&=& {\sigma}^{-\beta} \binom{n+\beta-1}{n}\EE[ g_t(n+B_{n+\beta,1-e^{-t}})]\eq
where $B_{n+\beta,1-e^{-t}}$ is a negative binomial random variable of parameters $n+\beta$ and $1-e^{-t}$.
Note that the moment generating function associated to $B_{n+\beta,1-e^{-t}}$ is well-known:
\bqn{NB}
\nonumber\fo s<-\ln(1-e^{-t}),\qquad \EE[\exp(sB_{n+\beta,1-e^{-t}})]&=&\lt(\frac{1-(1-e^{-t})}{1-(1-e^{-t})\exp(s)}\rt)^{n+\beta}\\
&=&\lt(\frac{e^{-t}}{1-(1-e^{-t})\exp(s)}\rt)^{n+\beta} \hskip-3mm. \eqn
Recalling the form of $g_t$ and taking into account that for $t>0$, $-t<0< -\ln(1-e^{-t}$, we get
\bq
\discret{D}^*_{e^{-t}} g_t(n)&=& e^{t\beta} \binom{n+\beta-1}{n}\lt(1- e^{-nt}\lt(\frac{1}{e^t+e^{-t}-1}\rt)^{n+\beta}\rt).
\eq
For fixed $n\in\ZZ_+$, this expression converges to zero as $t$ goes to zero.
Since $f\in\esFf$, this is sufficient to deduce that $\fm \discret{D}^*_{e^{-t}} g_t f^2$ converges to zero as $t$ goes to zero
and finally that \eqref{contd2} is satisfied.\par\sm
We
compute the generator $\discret{D}$ of the semi-group $(\discret{D}_{e^{-t}})_{t\geq 0}$ in a similar way.
Indeed, it is sufficient to describe its action on $\esFf$, since $\discret{D}$ is the closure of its restriction to this domain.
So fix $f\in\esFf$, we want to show that with $\discret{D} f$ defined as in \eqref{dD}, we have
\bqn{dD2}
\lim_{t\ri 0_+} \sum_{n\in\ZZ_+} \lt(\frac{(\discret{D}_{e^{-t}}-\discret{D}_{1})f(n)}{t}-\discret{D}f(n)\rt)^2\,\fm(n)&=&0.\eqn
To prove this convergence, we write that
for given $n\in\ZZ_+$, we have
\bq
\frac{(\discret{D}_{e^{-t}}-\discret{D}_{1})f(n)}{t}-\discret{D} f (n)&=&V_t f(n)+W_t f(n)\eq
with
\bqn{Vtfn}
\nonumber V_t f (n)&\df&\frac{( e^{-nt}-1)f(n)+ne^{-(n-1)t}(1-e^{-t})f(n-1)}{t}\\ && +\frac{n(n-1)e^{-(n-2)t}(1-e^{-t})^2f(n-2)/2}{t}\\&&-n(f(n-1)-f(n))\\
\nonumber W_t f (n)&\df&\frac1t\sum_{m=0}^{n-3}\binom{n}{m} e^{-mt}(1-e^{-t})^{n-m}f(m).\eqn
As in the proof of the continuity of the semi-group $(\discret{D}_{e^{-t}})_{t\geq 0}$, we will get \eqref{dD2}
as soon as we show that
\bqn{dD3}
\lim_{t\ri 0_+} \sum_{n\in\ZZ_+} \lt(V_t f(n)\rt)^2\,\fm(n)&=&0\\
\label{dD3b}\lim_{t\ri 0_+}\sum_{n\in\ZZ_+} \lt(W_t f (n)\rt)^2\,\fm(n)&=&0.\eqn
Again the sum in \eqref{dD3} contains only a finite number of terms, so it is sufficient to show that
for any $n\in\ZZ_+$, $\lim_{t\ri0_+} V_t f(n)=0$, convergence which is immediate in view of \eqref{Vtfn}.
Of course, \eqref{dD3} holds if $V_t f (n)$ is replaced by the more natural expression
\bq
\frac{( e^{-nt}-1)f(n)+ne^{-(n-1)t}(1-e^{-t})f(n-1)}{t}-n(f(n-1)-f(n))\eq
but \eqref{dD3b} is no longer true when $W_t f (n)$ is replaced by $\frac1t\sum_{m=0}^{n-2}\binom{n}{m} e^{-mt}(1-e^{-t})^{n-m}f(m)$.

To show \eqref{dD3b}, as before, we resort to Cauchy-Schwarz inequality: for any $t>0$ and any $n\in\ZZ_+$
\bq
\lt(W_t f (n)\rt)^2&\leq & h_t(n)\discret{D}_{e^{-t}} f^2 (n)\eq
with
\bq
h_t(n)&=&\frac1{t^2}\sum_{m=0}^{n-3}\binom{n}{m} e^{-mt}(1-e^{-t})^{n-m}\\
&=&\frac1{t^2}\lt(1-e^{-nt}-ne^{-(n-1)t}(1-e^{-t})-\frac{n(n-1)}{2}e^{-(n-2)t}(1-e^{-t})^2\rt)\\
&=&\frac1{t^2}\lt(1-e^{-nt}+\frac{(e^{t}-1)(e^t-3)}{2}ne^{-nt}-\frac{(e^{t}-1)^2}{2}n^2e^{-nt}\rt).\eq
It is thus sufficient to show that \bq
\lim_{t\ri 0_+}\fm h_t\discret{D}_{e^{-t}}f^2 &=&0\eq
to get \eqref{dD3b}. Since $f\in\esFf$, we just need to prove that for any $n\in\ZZ_+$,
\bq
\lim_{t\ri0_+} \discret{D}^*_{e^{-t}} h_t(n)&=&0\eq
and taking into account Lemma \ref{dualdD}, this amounts to
\bqn{htnB}
\lim_{t\ri0_+}\EE[ h_t(n+B_{n+\beta,1-e^{-t}})]&=&0\eqn
where $B_{n+\beta,1-e^{-t}}$ is still a negative binomial random variable of parameters $n+\beta$ and $1-e^{-t}$.
We compute, writing $\hbar_n(t)=t^2\EE[ h_t(n+B_{n+\beta,1-e^{-t}})]$, that
\bqn{deb}
\hbar_n(t)
&=&1-e^{-nt}\EE[ \exp(-tB_{n+\beta,1-e^{-t}})] \nonumber \\ && +
\frac{(e^{t}-1)(e^t-3)}{2}
e^{-nt}\EE[ (n+B_{n+\beta,1-e^{-t}})\exp(-tB_{n+\beta,1-e^{-t}})]\\
\nonumber&&-\frac{(e^{t}-1)^2}{2}e^{-nt}\EE[ (n+B_{n+\beta,1-e^{-t}})^2\exp(-tB_{n+\beta,1-e^{-t}})]\\
\nonumber&=&1-e^{-nt}\Bigg\{\lt(1-\frac{(e^{t}-1)(e^t-3)}{2}n+\frac{(e^{t}-1)^2}{2}n^2\rt)\EE[ \exp(-tB_{n+\beta,1-e^{-t}})]\\
\nonumber&&+\lt((e^{t}-1)^2n-\frac{(e^{t}-1)(e^t-3)}{2}\rt)\EE[ B_{n+\beta,1-e^{-t}}\exp(-tB_{n+\beta,1-e^{-t}})]\\
\nonumber&&-\frac{(e^{t}-1)^2}{2}n^2\EE[ B^2_{n+\beta,1-e^{-t}}\exp(-tB_{n+\beta,1-e^{-t}})]\Bigg\}\\
 \nonumber&=& 1-e^{-nt}\Bigg\{\lt(1+(e^{t}-1)n+\frac{(e^{t}-1)^2}{2}(n^2-n)\rt)\EE[ \exp(-tB_{n+\beta,1-e^{-t}})]\\
 \nonumber&&+\lt(e^t-1+(e^{t}-1)^2(n-1/2)\rt)\EE[ B_{n+\beta,1-e^{-t}}\exp(-tB_{n+\beta,1-e^{-t}})]\\
 \nonumber&&-\frac{(e^{t}-1)^2}{2}n^2\EE[ B^2_{n+\beta,1-e^{-t}}\exp(-tB_{n+\beta,1-e^{-t}})]\Bigg\}.
\eqn
\par
Differentiating \eqref{NB} with respect to $s<-\ln(1-e^{-t})$, we get, writing $E_n(s,t)=\EE[B_{n+\beta,1-e^{-t}}\exp(sB_{n+\beta,1-e^{-t}})]$, that
\bq
E_n(s,t)&=&\partial_s\lt(\frac{e^{-t}}{1-(1-e^{-t})\exp(s)}\rt)^{n+\beta}\\
&=&(n+\beta)\lt(\frac{e^{-t}}{1-(1-e^{-t})\exp(s)}\rt)^{n+\beta-1}\frac{e^{-t}(1-e^{-t})\exp(s)}{(1-(1-e^{-t})\exp(s))^2}\\
&=&(n+\beta)\frac{(1-e^{-t})\exp(s)}{1-(1-e^{-t})\exp(s)}\EE[\exp(sB_{n+\beta,1-e^{-t}})]\\
&=&(n+\beta)\lt(\frac{1}{1-(1-e^{-t})\exp(s)}-1\rt)\EE[\exp(sB_{n+\beta,1-e^{-t}})]
\eq
and, writing $\overline{E}_n(t,s)=\EE[B^2_{n+\beta,1-e^{-t}}\exp(sB_{n+\beta,1-e^{-t}})]$,
\bq\overline{E}_n(t,s)&=&\partial_s \EE[B_{n+\beta,1-e^{-t}}\exp(sB_{n+\beta,1-e^{-t}})]\\
&=&(n+\beta)\lt(\frac{(1-e^{-t})\exp(s)}{(1-(1-e^{-t})\exp(s))^2}\EE[\exp(sB_{n+\beta,1-e^{-t}})]\rt. \\ && +
\lt.\frac{(1-e^{-t})\exp(s)}{1-(1-e^{-t})\exp(s)}\pa_s \EE[\exp(sB_{n+\beta,1-e^{-t}})]\rt)\\
&=&(n+\beta)\frac{(1-e^{-t})\exp(s)}{1-(1-e^{-t})\exp(s)}\lt(\frac{1}{1-(1-e^{-t})\exp(s)}\rt. \\ && +
\lt.(n+\beta)\lt(\frac{1}{1-(1-e^{-t})\exp(s)}-1\rt)\rt)\EE[\exp(sB_{n+\beta,1-e^{-t}})]\\
&=&(n+\beta)\frac{(1-e^{-t})\exp(s)}{1-(1-e^{-t})\exp(s)}\lt(\frac{n+\beta+1}{1-(1-e^{-t})\exp(s)}-(n+\beta)\rt)\\ &&\EE[\exp(sB_{n+\beta,1-e^{-t}})]\\
&=&(n+\beta)\frac{(1-e^{-t})\exp(s)}{(1-(1-e^{-t})\exp(s))^2}\lt(1-(n+\beta)(1-(1-e^{-t})\exp(s))\rt)\\ && \EE[\exp(sB_{n+\beta,1-e^{-t}})]
. \eq
Taking $s=-t$ and writing $\epsilon\df 1-e^{-t}$,
it appears that,
\bq
\EE[\exp(-tB_{n+\beta,1-e^{-t}})]
&=&\lt(\frac{1}{e^t+e^{-t}-1}\rt)^{n+\beta}\\
&=&(1+\epsilon^2+o(\epsilon^2))^{-(n+\beta)}\\
&=&1-(n+\beta)\epsilon^2+o(\epsilon^2)\\
\EE[B_{n+\beta,1-e^{-t}}\exp(tB_{n+\beta,1-e^{-t}})]&=&(n+\beta)\frac{1-e^{-t}}{e^t+e^{-t}-1}\EE[\exp(-tB_{n+\beta,1-e^{-t}})]\\
&=&(n+\beta)\epsilon +o(\epsilon)\\
\EE[B^2_{n+\beta,1-e^{-t}}\exp(tB_{n+\beta,1-e^{-t}})]&=&(n+\beta)(1-(n+\beta))\epsilon +o(\epsilon).\eq
Injecting these approximations in \eqref{deb}, we obtain
\bq
t^2\EE[ h_t(n+B_{n+\beta,1-e^{-t}})]&=& 1-(1-\epsilon)^n\Bigg\{\lt(1+\frac{\epsilon}{1-\epsilon}n+\frac{\epsilon^2}{2}(n^2-n)\rt)(1-(n+\beta)\epsilon^2)\\
&&+\lt(\frac{\epsilon}{1-\epsilon}+\epsilon^2(n-1/2)\rt)(n+\beta)\epsilon\\
&&-\frac{\epsilon^2}{2}n^2(n+\beta)(1-(n+\beta))\epsilon
+o(\epsilon^2)\Bigg\}\\
&=&1-(1-\epsilon)^n\Bigg\{\lt(1+n\epsilon+\frac{\epsilon^2}{2}(n^2+n)\rt)(1-(n+\beta)\epsilon^2)\\
&&+(n+\beta)\epsilon^2
+o(\epsilon^2)\Bigg\}\\
&=&1-\lt(1-n\epsilon+\frac{n(n-1)}{2}\epsilon^2\rt)\lt(1+n\epsilon+(n^2+n)\frac{\epsilon^2}{2}\rt)+o(\epsilon^2)\\
&=&o(\epsilon^2).
\eq
Since $\epsilon$ is equivalent to $t$ as the latter goes to zero, we deduce \eqref{htnB}.
\wwtbp
The previous computations show that it is not always convenient to work with $\esL^2$ spaces, since the assertion of Lemma \ref{appl2} seems obvious
from a martingale problem point of view.

\end{document}